\documentclass[11pt]{article}

\usepackage[margin=1in]{geometry}
\usepackage{amsmath,amssymb,amsthm,mathtools,mathrsfs}
\renewcommand{\leq}{\leqslant}
\renewcommand{\le}{\leqslant}
\renewcommand{\geq}{\geqslant}
\renewcommand{\ge}{\geqslant}
\usepackage{microtype}
\usepackage{aliascnt}
\usepackage{enumitem}
\usepackage{array,booktabs,longtable}
\usepackage{cite}
\usepackage{xcolor}
\usepackage[colorlinks=true,linkcolor=blue!55!black,citecolor=blue!55!black,
  urlcolor=blue!55!black]{hyperref}
\hypersetup{
  pdftitle={Sharp Gaussian Divergence and Bernstein--Markov Inequalities},
  pdfauthor={Fan Chen, Sinho Chewi, Jianfeng Lu, Matthew S. Zhang}
}

\allowdisplaybreaks
\newtheorem{theorem}{Theorem}[section]
\newaliascnt{conjecture}{theorem}
\newtheorem{conjecture}[conjecture]{Conjecture}
\aliascntresetthe{conjecture}
\newaliascnt{proposition}{theorem}
\newtheorem{proposition}[proposition]{Proposition}
\aliascntresetthe{proposition}
\newaliascnt{lemma}{theorem}
\newtheorem{lemma}[lemma]{Lemma}
\aliascntresetthe{lemma}
\newaliascnt{corollary}{theorem}
\newtheorem{corollary}[corollary]{Corollary}
\aliascntresetthe{corollary}
\newaliascnt{remark}{theorem}
\newtheorem{remark}[remark]{Remark}
\aliascntresetthe{remark}

\usepackage[nameinlink,capitalize,noabbrev]{cleveref}
\crefname{conjecture}{Conjecture}{Conjectures}
\Crefname{conjecture}{Conjecture}{Conjectures}

\newcommand{\R}{\mathbb R}
\newcommand{\E}{\mathbb E}
\newcommand{\Pp}{\mathbb P}
\newcommand{\cN}{\mathcal N}
\newcommand{\ip}[2]{\langle #1,#2\rangle}
\newcommand{\norm}[1]{\lVert #1\rVert}
\newcommand{\abs}[1]{\lvert #1\rvert}
\newcommand{\1}{\mathbf 1}
\newcommand{\dd}{\mathrm d}
\newcommand{\deq}{\coloneqq}
\newcommand{\HS}{\mathrm{HS}}
\newcommand{\cH}{\mathcal H}
\newcommand{\Dom}{\operatorname{Dom}}

\newcommand{\Id}{\mathrm{Id}}

\title{Sharp Gaussian Divergence and Bernstein--Markov Inequalities}
\author{
 Fan Chen\thanks{Department of Electrical Engineering and Computer Science,
 Massachusetts Institute of Technology.
 Email: \href{mailto:fanchen@mit.edu}{\texttt{fanchen@mit.edu}}.}
 \and
 Sinho Chewi\thanks{Department of Statistics and Data Science, Yale University.
 Email: \href{mailto:sinho.chewi@yale.edu}{\texttt{sinho.chewi@yale.edu}}.}
 \and
 Jianfeng Lu\thanks{Department of Mathematics, Duke University.
 Email: \href{mailto:jianfeng@math.duke.edu}{\texttt{jianfeng@math.duke.edu}}.}
 \and
 Matthew S. Zhang\thanks{Department of Mathematics, Massachusetts Institute of Technology.
 Email: \href{mailto:shuns436@mit.edu}{\texttt{shuns436@mit.edu}}.}
}
\date{}

\begin{document}
\maketitle

\begin{abstract}
We prove two dimension-free estimates in Gaussian space.  The first is an
optimal Meyer-type inequality for the Gaussian divergence: for $p\geq2$,
\[
 \norm{\delta V}_{L^p(\gamma_d)}
 \leq \sqrt p\,\abs{\E V}
      +Cp\,\norm{DV}_{L^p(\gamma_d;\HS_d)}\, .
\]
Here $\HS_d$ is the space of $d\times d$ matrices with its Hilbert--Schmidt
norm.
Its main ingredient is a dimension-free weak-type $(1,1)$
bound for the second-order transform $D^2\mathcal N^{-1}$,
where $\mathcal N\deq-\Delta+x\cdot D$ is the Ornstein--Uhlenbeck operator. We also include a direct change-of-variables proof of the weaker bounded derivative divergence inequality in the appendix.

The second result is a Gaussian Bernstein--Markov inequality.  If $P$ is a polynomial, $p\geq2$, and
$m_p\deq(\E\abs g^p)^{1/p}$ for a standard Gaussian $g$, then
\[
 \norm{DP}_{L^p(\gamma_d;\R^d)}
 \leq\frac{2\sqrt{2e \deg P}}{m_p}\,\norm P_{L^p(\gamma_d)}\, .
\]
The factor $2$ is unnecessary when $p$ is even, or $P$ is even or odd. Thus, the constant is of order $\sqrt{\deg P/p}$ throughout the range $p\geq2$.
\end{abstract}

{\setlength{\parskip}{0pt}
\tableofcontents
}

\section{Introduction}
\label{sec:intro}

Consider the standard Gaussian measure
\[
 \gamma_d(\dd x)\deq(2\pi)^{-d/2}e^{-\abs{x}^2/2}\,\dd x
\]
on $\R^d$, and write $\E$ for integration
against $\gamma_d$.  Let $D$ denote the gradient, with
$D_j\deq\partial_j\deq\partial/\partial x_j$, and let
$\Delta\deq\sum_{j=1}^d\partial_j^2$ be the Laplacian.
A vector field
$V:\R^d\to\R^d$ has the derivative matrix $DV\deq(D_jV_i)_{i,j=1}^d$ and
the Gaussian divergence
\[
 \delta V(x)\deq x\cdot V(x)-\operatorname{tr}(DV(x))\, .
\]
We measure $DV$ in the Hilbert--Schmidt (Frobenius) norm,
\[
 \abs{DV}_{\HS}^2\deq\sum_{i,j=1}^d\abs{D_jV_i}^2\, ,
\]
and write $\HS_d$ for $\R^{d\times d}$ with this norm.
Gaussian integration by parts shows that $\delta$ is the adjoint of $D$ in
$L^2(\gamma_d)$.  The composition
\[
 \cN\deq\delta D=-\Delta+x\cdot D
\]
is the Ornstein--Uhlenbeck operator.  It is self-adjoint on $L^2(\gamma_d)$, its eigenvalues are
the integers $0,1,2,\dots$ with respective eigenspaces the Wiener chaoses of integer order, and
$-\cN$ generates the Ornstein--Uhlenbeck semigroup.

These are the finite-dimensional instances of the Malliavin derivative, the Skorokhod integral, and the number operator of the Malliavin calculus \cite{Nua06Malliavin}, and we use the two vocabularies interchangeably, while using the notation from Malliavin analysis.  Precise definitions, the Sobolev spaces on which $D$, $\delta$, and $\cN$ act, and their elementary facts are collected in \cref{sec:preliminaries}.

The purpose of this paper is to prove the following two results.

\begin{theorem}[Sharp Meyer-type divergence inequality]\label{thm:main}
There is a universal constant $C<\infty$ such that, for every $d\geq1$,
every $p\geq2$, and every
$V\in W^{1,p}(\gamma_d;\R^d)$,
\begin{equation}\label{eq:main}
 \norm{\delta V}_{L^p(\gamma_d)}
 \leq \sqrt p\,\abs{\E V}
      +Cp\,\norm{DV}_{L^p(\gamma_d;\HS_d)}\, .
\end{equation}
One may take $C=48$.
Conversely, there is a universal constant $c>0$ with the following
property: if $d\geq1$, $p\geq2$, and $A,B\geq0$ are such that
\begin{equation}\label{eq:main-converse}
 \norm{\delta V}_{L^p(\gamma_d)}
 \leq A\,\abs{\E V}+B\,\norm{DV}_{L^p(\gamma_d;\HS_d)}
 \qquad\text{for every }V\in W^{1,p}(\gamma_d;\R^d)\, ,
\end{equation}
then $A\geq c\sqrt p$ and $B\geq cp$.
One may take $c=1/e$.
\end{theorem}

\begin{theorem}[Gaussian Bernstein--Markov inequality]
\label{thm:Bernstein-Markov}
Let $n\geq1$, let $P:\R^d\to\R$ be a polynomial of degree at most $n$, and
let $p\geq2$.  Then
\begin{equation}
 \norm{DP}_{L^p(\gamma_d;\R^d)}
 \leq\frac{2\sqrt{2en}}{m_p}\,
       \norm P_{L^p(\gamma_d)}\, ,
 \label{eq:Bernstein-Markov}
\end{equation}
where $m_p\deq\bigl(\E_{g \sim \gamma_1} \abs g^p\bigr)^{1/p}\asymp\sqrt p$
is the $L^p$ norm of a standard Gaussian.
Consequently,
\begin{equation}
 \norm{DP}_{L^p(\gamma_d;\R^d)}
 \leq2e\sqrt2\,\sqrt{\frac np}\,
       \norm P_{L^p(\gamma_d)}\, .
 \label{eq:Bernstein-Markov-order}
\end{equation}
The factor $2$ in \eqref{eq:Bernstein-Markov} and \eqref{eq:Bernstein-Markov-order} may be omitted if $p$ is an
even integer or if $P$ is even or odd.
\end{theorem}

The upper bound in~\cref{thm:Bernstein-Markov} is optimal in order $\sqrt{n/p}$ for $p\geq2$; the matching
lower bound is proved in \cref{prop:Bernstein-Markov-sharpness}.

\begin{remark}[Finite dimension and the Malliavin setting]
\label{rem:Malliavin}
Both theorems generalize to an abstract Wiener space because their constants are independent of $d$. The Cameron--Martin space replaces $\R^d$, and its Hilbert--Schmidt operators replace $\HS_d$ in the derivative norm. Smooth
cylindrical fields are dense in the Malliavin Sobolev space
$\mathbb D^{1,p}$ \cite[Chapter~1]{Nua06Malliavin}.  Applying
\eqref{eq:main} to differences of such fields shows that their divergences are Cauchy in $L^p$; the adjoint identity identifies the limit as the Skorokhod integral of the limiting field.  Finite-dimensional projections in the target give the same extension for \cref{cor:Hilbert-main}.

For \cref{thm:Bernstein-Markov}, finite sums of chaoses of order at most
$n$ are the $L^p$ closure of cylindrical polynomials of degree at most $n$;
this follows from the usual $L^2$ approximation and equivalence of finite
Gaussian moments on a fixed finite sum of chaoses
\cite[Chapters~5--6]{Janson1997}.  The estimate applied to differences makes
their gradients Cauchy in $L^p$, and closedness of the Malliavin derivative
identifies the limit.  We work on $\R^d$ throughout, where the arguments
are more transparent.
\end{remark}

\subsection{The divergence inequality and its consequences}
\label{sec:Meyer-intro}

Recall the Gaussian Riesz transform
\[
 R\deq D\cN^{-1/2}\, ,
\]
where the spectral powers of $\cN$
are defined in \cref{sec:D-delta-N}, in particular $\cN^{-1/2}$ vanishes on constants.
Meyer~\cite{Meyer1984} proved that, for every $1<p<\infty$, and every $f$ such that $\E f=0$,
\[
 c_p\,\norm f_{L^p(\gamma_d)}
 \leq\norm{Rf}_{L^p(\gamma_d;\R^d)}
 \leq C_p\,\norm f_{L^p(\gamma_d)}
\]
with constants independent of $d$.  Substituting $f\deq\cN^{1/2}g$ gives
the two-sided comparison between $Dg$ and $\cN^{1/2}g$, since
$R\cN^{1/2}g=Dg$.

The $L^2(\gamma_d)$ adjoint of the Gaussian Riesz transform
\[
 R^*=\cN^{-1/2}\delta
\]
connects these estimates to the divergence.  By duality, the $L^{p'}$ bound
for $R$ gives an $L^p$ bound for $R^*$, where $p\geq2$ and
$p'\deq p/(p-1)$.
To estimate $\delta V$ using Meyer's inequality, we calculate (see \cref{sec:Hermite})
\[
 \delta V=\cN^{1/2}R^*V=R^*(\cN+1)^{1/2}V\, ,
\]
where $\cN$ acts componentwise on $V$.  The adjoint bound therefore reduces
the problem to estimating $(\cN+1)^{1/2}V$.  The reverse Meyer comparison
controls its $L^p$ norm by those of $V$ and $DV$, yielding
\[
 \begin{aligned}
 \norm{\delta V}_{L^p(\gamma_d)}
 &\leq C_{p}\,\norm{(\cN+1)^{1/2}V}_{L^p(\gamma_d;\R^d)}\\
 &\leq C_p\,\bigl(\norm V_{L^p(\gamma_d;\R^d)}
               +\norm{DV}_{L^p(\gamma_d;\HS_d)}\bigr)\, ,
 \end{aligned}
\]
with no dependence on $d$.  The refinement in which
$\norm V_{L^p(\gamma_d;\R^d)}$ is replaced
by $\abs{\E V}$ already appears in the standard Malliavin calculus proof
\cite[Proposition~1.5.8]{Nua06Malliavin}.  Pronk and Veraar
\cite[Proposition~4.6]{PronkVeraar2014} proved a corresponding qualitative estimate for
UMD-valued integrands. Nguyen~\cite[Proposition~2.3]{Nguyen2022Skorohod} obtained an explicit $O(p^5)$ constant in the classical Skorokhod integral estimate involving both $\norm V_{L^p(\gamma_d;\R^d)}$ and $\norm{DV}_{L^p(\gamma_d;\HS_d)}$.  What is new in \cref{thm:main} is the
sharp universal dependence $\sqrt p$ and $Cp$ uniform in dimension.

Our argument begins with the full second-order Gaussian Riesz transform
\[
 R_2\deq D^2\cN^{-1}\, .
\]
A variational Calder\'on--Zygmund decomposition, built from an obstacle
problem for $\cN$, gives a dimension-free weak-type $(1,1)$ estimate for the
Hilbert--Schmidt norm of $R_2f$.  This supplies the bound at the endpoint
$q=1$ of the $L^q$ scale, where strong bounds fail.  Interpolation with the
$L^2$ bound yields $L^q$ estimates of order $(q-1)^{-1}$ for $1<q\leq2$.
The shifted-transform identity, verified on each Wiener chaos, transfers
these estimates to the shifted second-order transform $(\cN+1)^{-1}D^2$.
Duality with $q=p'$ then proves \cref{thm:main}.  Working directly with
$R_2$ avoids the second factor of $(q-1)^{-1}$ that would arise from
iterating first-order estimates.

Earlier weak-type $(1,1)$ estimates for second-order Gaussian
Riesz transforms control individual coordinates with constants that may
depend on $d$
\cite{GarciaCuervaMauceriSjogrenTorrea1999,CasarinoCiattiSjogren2021}.
Our estimate controls the full Hilbert--Schmidt norm with a universal
constant.

The variational Calder\'on--Zygmund decomposition adapts the
obstacle-problem method of Ouyang, Spector, and Stockdale
\cite{OuyangSpectorStockdale2026}, who proved a dimension-free weak-type
$(1,1)$ estimate for the Euclidean vector Riesz transform.  Our construction
works directly with the Gaussian Dirichlet form.

Our proof also gives the following extension without any dependence on the target
dimension.

\begin{corollary}[Finite-dimensional Hilbert targets]
\label{cor:Hilbert-main}
Let $\cH$ be a finite-dimensional real Hilbert space, let $p\geq2$, and let
$V=(V_i)_{i=1}^d$ with
$V_i\in W^{1,p}(\gamma_d;\cH)$.  Then
\begin{equation}
 \norm{\delta V}_{L^p(\gamma_d;\cH)}
 \leq \sqrt p\,\norm{\E V}_{\cH^d}
      +Cp\,
       \norm{DV}_{L^p(\gamma_d;\cH^{d\times d})}\, .
 \label{eq:Hilbert-main}
\end{equation}
One may take $C=48$ in the second term.
\end{corollary}

Iterating the $L^q$ estimate for $R_2$ gives sharp bounds for all even-order
Gaussian Riesz transforms.  For $k\geq1$, let $D^kf$ denote the full
$k$-tensor of ordered derivatives,
with the Hilbert--Schmidt norm on $(\R^d)^{\otimes k}$.

\begin{corollary}[Sharp even-order Riesz transform estimates]
\label{cor:even-Riesz}
For every integer $r\geq1$, there are constants $0<c_r\leq C_r<\infty$ such
that, for every $d\geq1$ and $1<q\leq2$,
\begin{equation*}
 \frac{c_r}{(q-1)^r}
 \leq
 \norm{D^{2r}\cN^{-r}}_{L^q(\gamma_d)
   \to L^q(\gamma_d;(\R^d)^{\otimes 2r})}
 \leq\frac{C_r}{(q-1)^r}\, .
\end{equation*}
One may take $C_r=24\cdot48^{r-1}$.
The upper bound remains valid for finite-dimensional Hilbert-valued inputs,
independently of the target dimension.  Equivalently, if $p\geq2$, if
$\delta^{2r}$ denotes the iterated adjoint of $D^{2r}$, and if $A$ is a
smooth compactly supported $2r$-tensor field, then
\begin{equation}
 \norm{\cN^{-r}\delta^{2r}A}_{L^p(\gamma_d)}
 \leq C_r\,(p-1)^r\,
       \norm A_{L^p(\gamma_d;(\R^d)^{\otimes 2r})}\, ,
 \label{eq:even-Riesz-adjoint}
\end{equation}
and the order $p^r$ is optimal.
\end{corollary}

\Cref{cor:even-Riesz}, proved in \cref{sec:even-Riesz}, gives the sharp
$q\downarrow1$ estimate for the full even-order transforms.
Pisier's calculation \cite{Pisier1988} treats odd orders,
while the later higher-order estimates
\cite{GutierrezSegoviaTorrea1996,ForzaniScottoUrbina2001} do not track this
optimal dependence for the full tensor norm.  The comparison is developed
in \cref{app:related-work}.

An independent change-of-variables proof of a weaker bounded derivative
version of \cref{thm:main} is given in \cref{app:bounded-derivative}.

\subsection{The Gaussian Bernstein--Markov inequality and its consequences}
\label{sec:Bernstein-Markov-intro}

The classical inequalities of Markov and Bernstein bound the
derivative of a polynomial in terms of its degree and its size.  Markov
\cite{Markov1890} proved that every non-constant real polynomial $P$ on
$\R$ satisfies
\[
 \norm{DP}_{L^\infty([-1,1];\R)}
 \leq(\deg P)^2\,\norm P_{L^\infty([-1,1])}\, ,
\]
while Bernstein's inequality \cite{Bernstein1912} gives
\[
 \sqrt{1-x^2}\,\abs{DP(x)}
 \leq(\deg P)\,\norm P_{L^\infty([-1,1])}\, ,
 \qquad -1<x<1\, .
\]
Thus, the derivative bound depends linearly on the degree at each fixed
interior point.  At the ends of the interval, Chebyshev polynomials attain
the quadratic dependence in Markov's inequality.
Hille, Szeg\H{o}, and
Tamarkin \cite{HilleSzegoTamarkin1937} extended Markov's inequality to
$L^p[-1,1]$.  An extensive theory developed from these results; see Borwein and Erd\'elyi
\cite{BorweinErdelyi1995}.  To pass from a bounded interval to the whole
line or to $\R^d$, one must integrate polynomials against a weight, and the
Gaussian measure is the natural choice.

For the Gaussian measure on the line, Freud's inequality
\cite{Freud1971}, after rescaling, gives
\[
 \norm{DP}_{L^p(\gamma_1;\R)}
 \leq C\sqrt{\frac{\deg P}{p}}\,\norm P_{L^p(\gamma_1)}\, ,
 \qquad 1\leq p<\infty\, ,
\]
with a universal constant.  Freud subsequently treated more general
exponential weights \cite{Freud1977}.  Levin and Lubinsky
\cite{LevinLubinsky1994} proved general $L^p$ Bernstein--Markov inequalities
for such weights.  Our question is whether the Gaussian estimate
extends to polynomials on $\R^d$, with the
Euclidean norm of the full gradient on the left and a constant independent
of $d$.

An argument attributed to Maurey and Pisier by Eskenazis and Ivanisvili
\cite{EskenazisIvanisvili2020} gives
\[
 m_p\,\norm{DP}_{L^p(\gamma_d;\R^d)}
 \leq(\deg P)\,\norm P_{L^p(\gamma_d)}\, ,
 \qquad 1\leq p<\infty\, .
\]
This has the desired independence of $d$, but a linear dependence on the
degree. Eskenazis and Ivanisvili also raised the following conjecture:
\begin{conjecture}[Eskenazis--Ivanisvili {\cite[Question~3]{EskenazisIvanisvili2020}}]
\label{conj:Eskenazis-Ivanisvili}
For every $1\leq p<\infty$, there is a constant $C_p<\infty$, depending only on
$p$, such that
\[
 \norm{DP}_{L^p(\gamma_d;\R^d)}
 \leq C_p\sqrt{\deg P}\,\norm P_{L^p(\gamma_d)}
\]
for every $d\geq1$ and every non-constant polynomial $P:\R^d\to\R$.
\end{conjecture}

Eskenazis and Ivanisvili proved the weaker bound
\[
 \norm{DP}_{L^p(\gamma_d;\R^d)}
 \leq C_p\,(\deg P)^{\frac12+\frac1\pi
 \arctan\bigl(\frac{\abs{p-2}}{2\sqrt{p-1}}\bigr)}
 \,\norm P_{L^p(\gamma_d)}\, .
\]
Their proof first estimates $\cN^{1/2}P$ using complex
hypercontractivity, then applies Meyer's first-order comparison to obtain
the gradient bound.

More recently, Kosov~\cite{Kosov2026} proved the bound with
$\sqrt{\deg P}$ for even integers $p\geq4$, with constant $3\sqrt p$.
For all real $p\geq4$, he obtained a bound with degree factor
$(\deg P)^{1/2+\theta_p}$, where $0\leq\theta_p\leq2/(3p)$ and
$\theta_p=0$ when $p$ is even.  At even exponents, his proof uses
Gaussian integration by parts and an $L^2$ polynomial estimate applied to
integer powers of polynomials.

\Cref{thm:Bernstein-Markov} proves \cref{conj:Eskenazis-Ivanisvili} for
every $p\geq2$ and also identifies the optimal constant dependence on $p$ in this
range.

Our proof in \cref{sec:proof-Bernstein-Markov} estimates the
gradient directly through complex Gaussian rotations.  A moment recursion
in the Fock space gives a comparison for homogeneous holomorphic
polynomials, valid for every real $p\geq2$.  Homogenization using
auxiliary Gaussian variables extends it to the even and odd parts of an
arbitrary polynomial, and Cauchy's formula gives the gradient estimate.
This proof is independent of the divergence inequality in \cref{thm:main}.

For $1<p<2$, the dimension-free estimate with
$\sqrt{\deg P}$ remains open.  The product polynomials
$P_N(x)\deq\prod_{j=1}^Nx_j$ show that its constant must be at least of
order $(p-1)^{-1/2}$.  The expected estimate is
\[
 \norm{DP}_{L^p(\gamma_d;\R^d)}
 \leq C\sqrt{\frac{\deg P}{p-1}}\,
       \norm P_{L^p(\gamma_d)}\, .
\]
This also rules out the original conjecture for $p = 1$. Both the lower bound
and this limitation of our proof are discussed in \cref{sec:BM-tightness}.

The two main inequalities can be combined when each component of
a vector field is a polynomial.  Writing this field as
$P\deq (P_1,\dots,P_d)$, we obtain the following estimate for its Gaussian
divergence.

\begin{corollary}[Polynomial vector fields]
\label{cor:polynomial-divergence}
Let $n\geq1$, $p\geq2$, and $P=(P_1,\dots,P_d):\R^d\to\R^d$,
where each $P_i$ is a polynomial of degree at most $n$.
Then
\begin{equation}
 \norm{\delta P}_{L^p(\gamma_d)}
 \leq \sqrt p\,\abs{\E P}
      +96\sqrt{2e}\,p\sqrt n\,
       \norm{P-\E P}_{L^p(\gamma_d;\R^d)}\, .
 \label{eq:polynomial-divergence}
\end{equation}
In particular,
\[
 \norm{\delta P}_{L^p(\gamma_d)}
 \leq (1+192\sqrt{2e}) \,p\sqrt n\,\norm P_{L^p(\gamma_d;\R^d)}\, .
\]
For $n=0$, one has
$\norm{\delta P}_{L^p(\gamma_d)}=m_p\,\abs P\leq\sqrt p\,\abs P$.
\end{corollary}

To apply \cref{thm:main}, the proof first uses the scalar
Bernstein--Markov inequality on $\ip{G}{P-\E P}$, where $G$ is an
auxiliary standard Gaussian vector.  Gaussian randomization then gives
\[
 \norm{DP}_{L^p(\gamma_d;\HS_d)}
 \leq 2\sqrt{2en}\,\norm{P-\E P}_{L^p(\gamma_d;\R^d)}\, .
\]
In this passage to a vector-valued polynomial, the scalar factor
$m_p^{-1}\asymp p^{-1/2}$ is canceled by the factor $m_p$ from the
randomization.  Multiplying this derivative bound by the factor $p$ in
\cref{thm:main} gives the order $p\sqrt n$ in the corollary.
The derivative bound above cannot retain an additional factor $p^{-1/2}$
uniformly in $d$: for $P(x)=x$, its degree is $1$, its mean is zero, and
\[
 \norm{DP}_{L^p(\gamma_d;\HS_d)}=\sqrt d\, ,
 \qquad
 \norm P_{L^p(\gamma_d;\R^d)}\asymp\sqrt{d+p}\, .
\]
Taking $d=\lceil p\rceil$ would force $\sqrt p\lesssim1$ if that extra
factor were available.  This explains the loss in the derivative estimate;
we do not claim that the $p$ dependence of the corollary itself
is optimal.

Further comparisons with Gaussian Riesz transform theory and
concentration estimates are collected in \cref{app:related-work}.

\paragraph{AI usage.}
In our prior work~\cite{Chen+26PicardHMC}, we established \cref{thm:main} but with constant $Cp^2$ in front of the second term, rather than the optimal constant $Cp$; see Lemma C.3 therein. Also, by applying the result of~\cite{EskenazisIvanisvili2020}, we established \cref{cor:polynomial-divergence} but with constant $C p^2 n$ in front of the second term, rather than $Cp\sqrt n$; see Lemma C.8 therein. This work arose from our efforts to find refined versions of these inequalities.
The original proof ideas were found through interactions with GPT-5.6 Sol. The authors recast the arguments, prepared the manuscript, and take full responsibility for its contents.

\section{Preliminaries}
\label{sec:preliminaries}

We work on $\R^d$ with $d$ finite and arbitrary.  This section collects
the facts about $D$, $\delta$, $\cN$, and the Wiener chaoses used in the
proofs.  Further analytic material appears in \cref{app:analytic-tools}.
The complex Gaussian measures and the Fock space are introduced in
\cref{sec:Fock} for the proof of \cref{thm:Bernstein-Markov}.

\subsection{Notation and conventions}
\label{sec:notation}

Throughout, $\gamma_m$ denotes the standard Gaussian measure on $\R^m$,
and $X$, $Y$, $G$ denote
standard Gaussian vectors of the appropriate dimensions.  Independence is
specified when these vectors are introduced; a rotated pair may depend on
the original pair. The letters $C$ and $c$ denote positive universal constants
whose values may change from line to line; a dependence on a parameter is
indicated by a subscript, as in $C_r$.

\paragraph{Vector-valued spaces.}
If $\cH$ is a finite-dimensional real Hilbert space, $L^p(\gamma_d;\cH)$ is the
Bochner space of $\cH$-valued functions with norm
$\norm F_{L^p(\gamma_d;\cH)}\deq(\E\norm F_\cH^p)^{1/p}$; inside proofs we
write $\norm F_{L^p(\cH)}$ or $\norm F_p$ when the measure and target are
clear.  We write $\cH^d$ and $\cH^{d\times d}$ for the
spaces of $d$-tuples and of $d\times d$ arrays of elements of $\cH$, with
the norms
\[
 \norm{(F_j)}_{\cH^d}^2\deq\sum_{j=1}^d\norm{F_j}_\cH^2\, ,
 \qquad
 \norm{(A_{ij})}_{\cH^{d\times d}}^2\deq\sum_{i,j=1}^d\norm{A_{ij}}_\cH^2\, .
\]
In particular $\HS_d\deq\R^{d\times d}$ carries the Hilbert--Schmidt norm
$\abs A_{\HS}^2\deq\sum_{i,j}A_{ij}^2$. The derivative $DF\deq(D_jF)_{j=1}^d$ of
a $\cH$-valued $F$ takes values in $\cH^d$, with
$\abs{DF}^2\deq\sum_j\norm{D_jF}_\cH^2$, and the derivative
$DV\deq(D_jV_i)_{i,j=1}^d$ of a field $V=(V_i)_{i=1}^d$ with $\cH$-valued
components takes values in $\cH^{d\times d}$.  For $k\geq1$, $D^kf$ is the
full $k$-tensor of ordered partial derivatives, measured in the
Hilbert--Schmidt norm of $(\R^d)^{\otimes k}$.  The operators $D$, $\delta$,
$\cN$, the Ornstein--Uhlenbeck semigroup, and its resolvents act on
$\cH$-valued functions componentwise in an orthonormal basis of $\cH$; the
resulting operators do not depend on the basis.

For $1\leq p<\infty$, both $C_{\rm c}^\infty(\R^d;\cH)$ and the $\cH$-valued
polynomials are dense in $L^p(\gamma_d;\cH)$.  If $1<p<\infty$ and $p'\deq p/(p-1)$ is the dual exponent, the dual of
$L^{p'}(\gamma_d;\cH)$ is $L^p(\gamma_d;\cH)$ under the pairing $\E\ip FG_\cH$; this is
elementary because $\dim \cH<\infty$.

\paragraph{Sobolev spaces.}
Let $1\leq p<\infty$.  A function $u\in L^p(\gamma_d;\cH)$ belongs to
$W^{1,p}(\gamma_d;\cH)$ if its distributional gradient on $\R^d$ is a function
with $Du\in L^p(\gamma_d;\cH^d)$, and then
$\norm u_{W^{1,p}}\deq\norm u_p+\norm{Du}_p$.
Polynomials, $C_{\rm c}^\infty$ functions, and bounded smooth functions with
bounded derivatives are all dense in $W^{1,p}(\gamma_d;\cH)$
\cite[Chapter~5]{Bogachev1998}, so $W^{1,p}(\gamma_d;\cH)$ is also the
completion of any of these classes under $\norm\cdot_{W^{1,p}}$.
Higher-order spaces $W^{k,p}$
are defined in the same way using all derivatives up to order $k$.  Because $\gamma_d$ is a probability measure,
$W^{1,p}\subset W^{1,r}$ for $r\leq p$; in particular every field in
\cref{thm:main} lies in $W^{1,2}$.  For $V\in W^{1,p}(\gamma_d;\cH^d)$, the
divergence is defined distributionally as
\begin{equation}\label{eq:delta-distribution}
 \ip{\delta V}{F}\deq\E\ip{V}{DF}_{\cH^d}\, ,
 \qquad F\in C_{\rm c}^\infty(\R^d;\cH)\, ,
\end{equation}
and $\delta V$ agrees with $x\cdot V-\operatorname{tr}(DV)$ for smooth $V$ by
\cref{lem:IBP} below. \Cref{lem:chaos-Sobolev} identifies this with an
$L^2$ function when $V\in W^{1,2}$.

\subsection{Hermite polynomials and Wiener chaos}
\label{sec:Hermite}

The probabilists' Hermite polynomials are
\[
 H_k(x)\deq(-1)^ke^{x^2/2}\,\frac{\dd^k}{\dd x^k}\,e^{-x^2/2}\, ,
 \qquad k\geq0\, ,
\]
so that $H_0=1$, $H_1(x)=x$, $H_2(x)=x^2-1$, and in general the leading term of $H_k(x)$ is $x^k$
plus terms of lower same-parity degrees.  For a multi-index
$\alpha\in\mathbb N^d$ put $H_\alpha(x)\deq\prod_{i=1}^dH_{\alpha_i}(x_i)$,
$\alpha!\deq\prod_i\alpha_i!$, and $\abs\alpha\deq\sum_i\alpha_i$.  With the
complex-bilinear dot product $t\cdot x\deq\sum_{i=1}^d t_ix_i$, the Hermite
polynomials are characterized by the generating function
\begin{equation}\label{eq:Hermite-generating-d}
 \exp\Bigl\{t\cdot x-\frac{t\cdot t}{2}\Bigr\}
 =\sum_{\alpha\in\mathbb N^d}\frac{t^\alpha}{\alpha!}\,H_\alpha(x)\, ,
 \qquad t,x\in\mathbb C^d\, ,
\end{equation}
which converges absolutely and locally uniformly.  Computing
$\E\exp\{t\cdot X-t\cdot t/2+s\cdot X-s\cdot s/2\}=e^{t\cdot s}$ gives
$\E[H_\alpha H_\beta]=\alpha!\,\delta_{\alpha\beta}$, and the Hermite
polynomials are complete in $L^2(\gamma_d)$
\cite[Chapter~1]{Bogachev1998}, \cite[Chapter~2]{Janson1997}.

For $n\geq0$, the $n$-th \emph{Wiener chaos} is
\[
 \mathcal C_n\deq\operatorname{span}\{H_\alpha:\abs\alpha=n\}\, ,
\]
the span of the Hermite polynomials indexed by $\abs\alpha=n$.  Thus,
\begin{equation*}
 L^2(\gamma_d)=\bigoplus_{n\geq0}\mathcal C_n
\end{equation*}
is an orthogonal direct sum.  We write $J_n$ for the orthogonal projection
onto $\mathcal C_n$ and $u_n\deq J_nu$, so that $u=\sum_{n\geq0}u_n$ is the
\emph{chaos expansion} of $u\in L^2(\gamma_d)$; $J_0u=\E u$, and the
elements of $\mathcal C_1$ are the centered linear functions
$x\mapsto\ip ax$.  Since $H_\alpha$ is $x^\alpha$ plus terms of lower
degree, the polynomials of degree at most $n$ form exactly the space
$\bigoplus_{k\leq n}\mathcal C_k$; in particular every polynomial has a
finite chaos expansion, and a polynomial lying in a single $\mathcal C_n$ need
not be homogeneous as a polynomial.  For $\cH$-valued functions, the chaos
expansion is taken componentwise, and $\mathcal C_n(\cH)\deq\mathcal C_n\otimes \cH$.

Differentiating \eqref{eq:Hermite-generating-d} in $x_i$ and in $t_i$ gives
\begin{equation}\label{eq:Hermite-D-delta}
 D_iH_\alpha=\alpha_iH_{\alpha-e_i}\, ,
 \qquad
 x_iH_\alpha-D_iH_\alpha=H_{\alpha+e_i}\, ,
\end{equation}
where $e_i$ is the $i$-th unit vector and $H_{\alpha-e_i}\deq0$ when
$\alpha_i=0$.  In the notation of the introduction the second identity reads
$\delta(H_\alpha e_i)=H_{\alpha+e_i}$.  Consequently
\begin{equation}\label{eq:chaos-shifts}
 D:\mathcal C_n\to\mathcal C_{n-1}\otimes\R^d\, ,
 \qquad
 \delta:\mathcal C_n\otimes\R^d\to\mathcal C_{n+1}\, ,
 \qquad
 \cN H_\alpha=\delta DH_\alpha=\abs\alpha\,H_\alpha\, .
\end{equation}
Thus \emph{$D$ lowers and $\delta$ raises the chaos order by one}, and
$\cN$ is diagonal in the chaos decomposition.

Gaussian integration by parts gives the adjoint relation between $D$ and
$\delta$. The proof is standard and thus will be omitted. See \cite[Section~1.3]{Nua06Malliavin} for the infinite-dimensional
statement.

\begin{lemma}[Gaussian integration by parts]\label{lem:IBP}
Let $F,G\in C^1(\R^d)$ be such that $F$, $G$, $DF$, and $DG$ have at most
polynomial growth.  Then, for $1\leq i\leq d$,
\begin{equation}\label{eq:IBP}
 \E[(D_iF)\,G]=\E\bigl[F\,(x_iG-D_iG)\bigr]\, .
\end{equation}
Consequently $\E\ip{DF}{V}=\E[F\,\delta V]$ for such $F$ and for every
$C^1$ vector field $V$ with $V$ and $DV$ of polynomial growth.  The identity
\eqref{eq:IBP} remains true when $1<p<\infty$, $F\in W^{1,p}(\gamma_d)$, and
$G\in C^1(\R^d)$ with $G$ and $DG$ of polynomial growth.
\end{lemma}

\subsection{The gradient, the divergence, and the Ornstein--Uhlenbeck operator}
\label{sec:D-delta-N}

The chaos expansion gives the following descriptions of $D$, $\delta$,
and $W^{1,2}(\gamma_d)$.  These statements apply componentwise to
$\cH$-valued functions.  By the density result in \cref{sec:notation}, the
definition of $W^{1,2}(\gamma_d)$ by weak derivatives agrees with the closure
of smooth functions used in the references.

\begin{lemma}[Chaos description of $D$, $\delta$, and $W^{1,2}$]
\label{lem:chaos-Sobolev}
\begin{enumerate}[label=\textup{(\roman*)}]
\item Let $u\in W^{1,2}(\gamma_d)$.  Then, the chaos projections and gradient satisfy $J_{n-1}(D_iu)=D_i(J_nu)$ for every $n\geq1$ and every
$i$, and
\begin{equation}\label{eq:Du-chaos}
 \norm{Du}_{L^2(\gamma_d;\R^d)}^2
 =\sum_{n\geq1}n\,\norm{u_n}_{L^2(\gamma_d)}^2\, .
\end{equation}
Conversely, every $u\in L^2(\gamma_d)$ with
$\sum_nn\,\norm{u_n}_{L^2(\gamma_d)}^2<\infty$ belongs to $W^{1,2}(\gamma_d)$.  In
particular, the Gaussian Poincar\'e inequality
\begin{equation}\label{eq:Poincare}
 \norm{u-\E u}_{L^2(\gamma_d)}\leq\norm{Du}_{L^2(\gamma_d;\R^d)}\, ,
 \qquad u\in W^{1,2}(\gamma_d)\, ,
\end{equation}
holds, and it holds componentwise for $\cH$-valued $u$.
\item Let $V\in W^{1,2}(\gamma_d;\cH^d)$.  Then, the distribution $\delta V$ of
\eqref{eq:delta-distribution} is a function in $L^2(\gamma_d;\cH)$, one has
$\E\ip{F}{\delta V}_\cH=\E\ip{DF}{V}_{\cH^d}$ for every
$F\in W^{1,2}(\gamma_d;\cH)$, and
\begin{equation}\label{eq:delta-L2}
 \norm{\delta V}_{L^2(\gamma_d;\cH)}^2
 =\norm V_{L^2(\gamma_d;\cH^d)}^2
  +\E\sum_{i,j=1}^d\ip{D_iV_j}{D_jV_i}_\cH
 \leq\norm V_{L^2(\gamma_d;\cH^d)}^2
      +\norm{DV}_{L^2(\gamma_d;\cH^{d\times d})}^2\, .
\end{equation}
\item Let $D$ be regarded as an operator on $L^2(\gamma_d)$ with domain
$W^{1,2}(\gamma_d)$.  Then, $D$ is closed, and $\delta$, with domain
\[
 \Dom(\delta)
 \deq\bigl\{V\in L^2(\gamma_d;\R^d):
   \abs{\E\ip{DF}{V}}\leq C_V\,\norm F_{L^2(\gamma_d)}
   \text{ for all }F\in W^{1,2}(\gamma_d)\bigr\}\, ,
\]
is its Hilbert space adjoint.  In particular $\delta$ is a closed operator,
$W^{1,2}(\gamma_d;\R^d)\subset\Dom(\delta)$, and for $V\in\Dom(\delta)$ the
function $\delta V$ is the distribution \eqref{eq:delta-distribution}.
\end{enumerate}
\end{lemma}

For (i) and (ii) see \cite[Sections~1.2--1.3]{Nua06Malliavin}, where the
identity in (ii) is the formula for the second moment of the Skorokhod
integral; (iii) says that $\delta$ is the adjoint of the closed operator
$D$, see \cite[Section~1.3]{Nua06Malliavin}.

Since $\cN H_\alpha=\abs\alpha H_\alpha$ by \eqref{eq:chaos-shifts}, we
\emph{define} $\cN$ on $L^2(\gamma_d)$ spectrally:
\begin{equation*}
 \Dom(\cN)
 \deq\Bigl\{u\in L^2(\gamma_d):\sum_{n\geq0}n^2\,\norm{u_n}_2^2<\infty\Bigr\}\, ,
 \qquad
 \cN u\deq\sum_{n\geq0}nu_n\, .
\end{equation*}
Thus, $\cN$ is self-adjoint and non-negative, its spectrum is
$\{0,1,2,\dots\}$, the chaos $\mathcal C_n$ is the eigenspace for the
eigenvalue $n$, and the kernel of $\cN$ consists of the constants.  On smooth
functions with polynomially bounded derivatives, \cref{lem:IBP} shows that
$\cN u=\delta Du=-\Delta u+x\cdot Du$ is the classical Ornstein--Uhlenbeck
operator, and \cref{lem:chaos-Sobolev} identifies
$W^{1,2}(\gamma_d)=\Dom(\cN^{1/2})$ with $\norm{Du}_2=\norm{\cN^{1/2}u}_2$.
Functions of $\cN$ are defined by spectral calculus; in particular, for
$a>0$,
\begin{equation}\label{eq:resolvent-spectral}
 (\cN+a)^{-1}u\deq\sum_{n\geq0}\frac{u_n}{n+a}\, ,
 \qquad
 \cN^{-1}u\deq\sum_{n\geq1}\frac{u_n}{n}\, ,
\end{equation}
so that $\cN^{-1}$ vanishes on constants and inverts $\cN$ on their
orthogonal complement:
\begin{equation}\label{eq:N-inverse-N}
 \cN^{-1}\cN u=u-\E u\, ,
 \qquad u\in\Dom(\cN)\, .
\end{equation}
The powers $\cN^{-1/2}$ and $(\cN+1)^{1/2}$ mentioned in the introduction are
defined in the same way.

\begin{lemma}[Weak formulation and commutation relations]\label{lem:N-weak}
\begin{enumerate}[label=\textup{(\roman*)}]
\item For $u,\varphi\in W^{1,2}(\gamma_d)$,
\begin{equation*}
 \E\ip{Du}{D\varphi}=\sum_{n\geq1}n\,\E[u_n\varphi_n]\, .
\end{equation*}
Consequently, if $u\in W^{1,2}(\gamma_d)$ and $h\in L^2(\gamma_d)$, then
$u\in\Dom(\cN)$ and $\cN u=h$ if and only if
\begin{equation}\label{eq:N-weak}
 \E\ip{Du}{D\varphi}=\E[h\varphi]
 \qquad\text{for all }\varphi\in W^{1,2}(\gamma_d)\, ,
\end{equation}
and it suffices to verify \eqref{eq:N-weak} for polynomials $\varphi$.
Moreover $\Dom(\cN)\subset W^{1,2}(\gamma_d)$.
\item On polynomials,
\begin{equation*}
 D_i\, \cN=(\cN+1)D_i\, ,
 \qquad
 \cN\delta=\delta(\cN+1)\, .
\end{equation*}
More generally, if $\phi$ is a bounded function on $\{0,1,2,\dots\}$ and
$u\in W^{1,2}(\gamma_d)$, then $D\phi(\cN)u=\phi(\cN+1)Du$.
\end{enumerate}
\end{lemma}

Part (i) says that $\cN$ is the generator of the Dirichlet form of
$\gamma_d$ \cite[Chapter~5]{Bogachev1998}; the commutation relations in
(ii) are in \cite[Section~1.4]{Nua06Malliavin} and
\cite[Section~2.7]{BGL14}, and the general statement
follows from them chaos by chaos.

\subsection{The Ornstein--Uhlenbeck semigroup and its resolvents}
\label{sec:OU-semigroup}

For $t\geq0$ and $F\in L^1(\gamma_d;\cH)$, define the Ornstein--Uhlenbeck semigroup by Mehler's formula,
\begin{equation*}
 P_tF(x)\deq\E_YF\bigl(e^{-t}x+\sqrt{1-e^{-2t}}\,Y\bigr)\, ,
 \qquad Y\sim\gamma_d\, .
\end{equation*}

\begin{lemma}[Ornstein--Uhlenbeck semigroup]\label{lem:OU}
Let $\cH$ be a finite-dimensional Hilbert space and $1\leq q<\infty$.
\begin{enumerate}[label=\textup{(\roman*)}]
\item $P_t$ is well defined on $L^q(\gamma_d;\cH)$ and
$\norm{P_tF}_{L^q(\gamma_d;\cH)}\leq\norm F_{L^q(\gamma_d;\cH)}$.
On scalar functions, $P_t1=1$ and $P_tF\geq0$ whenever $F\geq0$.
The same assertions hold with $q=\infty$.
\item $P_sP_t=P_{s+t}$, and $P_tF\to F$ in $L^q(\gamma_d;\cH)$ as
$t\downarrow0$.
\item If $F\in W^{1,q}(\gamma_d;\cH)$, then $P_tF\in W^{1,q}(\gamma_d;\cH)$ and
$D(P_tF)=e^{-t}P_t(DF)$.
\item $P_tH_\alpha=e^{-\abs\alpha t}H_\alpha$ for every $\alpha$.  Hence
$P_t=e^{-t\cN}$ on $L^2(\gamma_d)$, $P_t$ is self-adjoint on $L^2$, and
$J_nP_t=e^{-nt}J_n$.
\item For $a>0$, the Bochner integral
\begin{equation*}
 (\cN+a)^{-1}=\int_0^\infty e^{-at}P_t\,\dd t
\end{equation*}
defines a bounded operator on $L^q(\gamma_d;\cH)$ of norm at most $1/a$,
which agrees on $L^2$ with the spectral resolvent
\eqref{eq:resolvent-spectral}.  In particular $(\cN+1)^{-1}$ is a
contraction on every $L^q(\gamma_d;\cH)$, and
\begin{equation}\label{eq:Ma}
 M_a\deq\cN(\cN+a)^{-1}=\Id-a(\cN+a)^{-1}
\end{equation}
satisfies
$\norm{M_a}_{L^q(\gamma_d;\cH)\to L^q(\gamma_d;\cH)}\leq2$.
\end{enumerate}
\end{lemma}

For these semigroup and resolvent properties, see \cite[Section~2.7]{BGL14},
\cite[Section~1.4]{Nua06Malliavin}, and \cite[Chapter~1]{Bogachev1998}.  The
bound on $M_a$ follows from \eqref{eq:Ma} and the bound $1/a$ on the
resolvent.

\section{Proof of the divergence inequality and its consequences}
\label{sec:Meyer-proof}

This section proves the upper bound \eqref{eq:main} in \cref{thm:main}, its
Hilbert-valued extension \cref{cor:Hilbert-main}, and \cref{cor:even-Riesz}.
The converse assertion of \cref{thm:main} is proved in \cref{sec:tightness}.

Throughout, $p\ge2$, and $p'\deq p/(p-1)\in(1,2]$ is its dual exponent, so that $(p'-1)^{-1}=p-1$.
The central operator is the second-order Gaussian Riesz
transform, initially defined on $L^2(\gamma_d)$ by
\begin{equation*}
 R_2f\deq D^2\cN^{-1}(f-\E f)\, .
\end{equation*}
Its value is a symmetric $d\times d$ matrix, measured in the
Hilbert--Schmidt norm.

Duality explains why the second-order transform enters the divergence
inequality.  For $V\in W^{1,2}(\gamma_d;\R^d)$ and a test function
$F\in C_{\rm c}^\infty(\R^d)$, if $\E V = 0$, we establish in \cref{sec:duality}
\begin{equation}\label{eq:centered-duality-overview}
 \ip{\delta V}{F}
 =\E\sum_{i,j=1}^d
   \bigl[(\cN+1)^{-1}D_jD_iF\bigr]D_jV_i\, .
\end{equation}
Writing $T\deq(\cN+1)^{-1}D^2$, this identity says
$\delta V=T^*(DV)$ in the distributional sense.  Thus we need an $L^p$
bound for the adjoint $T^*$, or equivalently an $L^{p'}$ bound for $T$:
\[
 \abs{\ip{\delta V}{F}}
 \leq\norm{DV}_{L^p(\gamma_d;\HS_d)}
       \,\norm{TF}_{L^{p'}(\gamma_d;\HS_d)}\, .
\]
For general (non-centered) $V$, the mean is handled separately as a Gaussian linear form
and contributes the factor $\sqrt p$.

We establish the required estimate in the following order.
\begin{enumerate}
\item \emph{Variational decomposition.}  \Cref{sec:CZ} constructs a
dimension-free variational Calder\'on--Zygmund decomposition from an obstacle
problem for $\cN$.

\item \emph{The weak-type estimate at $q=1$.}  Strong $L^q$ bounds for
$R_2$ blow up as $q\downarrow1$; at the endpoint $q=1$ the correct substitute
is a weak-type $(1,1)$ estimate, and \cref{sec:weak-endpoint} uses the
decomposition to prove one whose constant does not depend on $d$.
Interpolation with the $L^2$ contraction of $R_2$ then gives the required
order $(q-1)^{-1}$ on $L^q$.

\item \emph{The shifted transform and Hilbert targets.}  \Cref{sec:shift}
uses the chaos-by-chaos identity
\[
 (\cN+1)^{-1}D^2=[\Id+(\cN+1)^{-1}]R_2\, ,
\]
followed by Gaussian randomization, to obtain the Hilbert-valued estimate
with no dependence on the target dimension.

\item \emph{Duality.}  \Cref{sec:duality} combines this estimate with the
centered duality identity \eqref{eq:centered-duality-overview} and treats
the mean separately, completing the proof of the divergence inequality.

\item \emph{Even-order transforms.}  \Cref{sec:even-Riesz} iterates the
second-order bound and proves a matching lower bound, establishing
\cref{cor:even-Riesz}.
\end{enumerate}

\subsection{A variational Calder\'on--Zygmund decomposition}
\label{sec:CZ}

We construct a Calder\'on--Zygmund decomposition of a non-negative function
$f$ at a height $\lambda$.  Such a decomposition splits $f$ into a
part bounded by $\lambda$ and a remainder that
lives on a set of measure of order $\E f/\lambda$.  A weak-type estimate
for an operator such as $R_2$ then follows from two ingredients: an $L^2$
bound for the bounded part, and control of the remainder by the measure of
the set that carries it.  In the Euclidean setting the splitting is
geometric, by dyadic cubes, and in the Gaussian setting by the admissible
balls of Mauceri and Meda \cite{MauceriMeda2007}; both constructions
produce constants that depend on $d$.

We obtain the splitting instead from an obstacle problem for $\cN$, in a similar style to the obstable problem approach in \cite{OuyangSpectorStockdale2026}.
We minimize the Dirichlet energy with source $f-\lambda$ over non-negative
functions in $W^{1,2}(\gamma_d)$ and denote a minimizer by $u$.
The associated Lagrange multiplier determines the
bounded part $\mu$, while $\cN u$ is the remainder.  Because the remainder lies in
the range of $\cN$, its second-order Riesz transform is the Hessian
$D^2u$, which vanishes on the zero set $\{u=0\}$; the height $\lambda$
keeps the set $\{u>0\}$ small.
Only the $L^2$ theory of $\cN$ enters, which is why none of the constants
depends on $d$.

\begin{proposition}[Variational Calder\'on--Zygmund decomposition]\label{prop:obstacle}
Let $f\in L^\infty(\gamma_d)$ satisfy $f\ge0$, and let $\lambda>\E f$.  Then there exist
$u\in\Dom(\cN)$ and $\mu\in L^\infty(\gamma_d)$ such that
\begin{align}
 &u\ge0\, ,\qquad 0\le\mu\le\lambda\, ,\notag\\
 &f = \cN u + \mu\, ,\qquad \E\mu=\E f\, ,\notag\\
 &(\lambda-\mu)u=0\quad\text{a.e.}\label{eq:obstacle-complementarity}
\end{align}
With $\Omega\deq\{u>0\}$, one has
\begin{equation}\label{eq:obstacle-support}
 \mu=\lambda\quad\text{a.e. on }\Omega\, ,
 \qquad
 \gamma_d(\Omega)\le\frac{\E f}{\lambda}\, ,
 \qquad
 D^2u=0\quad\text{a.e. on }\Omega^{\mathrm c}\, .
\end{equation}
\end{proposition}

\begin{proof}
We replace the constraint $u\ge0$ by a penalty with a small parameter
$\varepsilon$, minimize the penalized energy, prove a pointwise lower bound
for the minimizers, and pass to the limit $\varepsilon\downarrow0$.

Put $g\deq f-\lambda$ and $a\deq-\E g=\lambda-\E f>0$.  For $\varepsilon\in(0,1)$
define the penalty and its derivative,
\[
 B_\varepsilon(s)\deq\frac1{2\varepsilon}\,(s^-)^2\, ,
 \qquad
 B_\varepsilon'(s)\deq-\frac{s^-}{\varepsilon}\, ,
 \qquad s^-\deq\max\{-s,0\}\, .
\]
The penalty $B_\varepsilon$ is convex and continuously differentiable.
Consider on $W^{1,2}(\gamma_d)$ the
penalized energy
\begin{equation*}
 J_\varepsilon(v)
 \deq\frac12\,\E|Dv|^2+\E B_\varepsilon(v)-\E(gv)\, .
\end{equation*}

\emph{Existence of a minimizer.}
Set $b\deq\norm{f-\E f}_2=\norm{g-\E g}_2$.  By the Gaussian Poincar\'e
inequality \eqref{eq:Poincare},
\begin{equation}\label{eq:penalized-source-bound}
 \E(gv)=\E\bigl[(g-\E g)(v-\E v)\bigr]+\E g\,\E v
 \le b\,\norm{Dv}_2-a\,\E v\, .
\end{equation}
If $\E v\ge0$, then
\[
 J_\varepsilon(v)
 \ge \frac12\,\norm{Dv}_2^2-b\,\norm{Dv}_2+a\,\E v\, .
\]
If $\E v<0$, Jensen's inequality gives
$\E(v^-)^2\ge(\E v^-)^2\ge(\E v)^2$, and therefore
\[
 J_\varepsilon(v)
 \ge \frac12\,\norm{Dv}_2^2-b\,\norm{Dv}_2+\frac{(\E v)^2}{2\varepsilon}+a\,\E v\, .
\]
These two estimates prove coercivity for fixed $\varepsilon$.  The functional is convex and weakly
lower semicontinuous, so the direct method gives a minimizer $u_\varepsilon$.  Its Euler--Lagrange equation in the weak form is
\begin{equation}\label{eq:penalized-equation}
 \E\ip{Du_\varepsilon}{D\varphi}
 +\E[B_\varepsilon'(u_\varepsilon)\varphi]
 =\E[g\varphi]
 \qquad \forall\,\varphi\in W^{1,2}(\gamma_d)\, .
\end{equation}

\emph{A pointwise lower bound for $u_\varepsilon$.}
Since $f\ge0$, one has $B_\varepsilon'(-\varepsilon\lambda)=-\lambda\le g$.
Use $\eta\deq(-\varepsilon\lambda-u_\varepsilon)^+\in W^{1,2}$ as a test function
in \eqref{eq:penalized-equation}.  Since $\eta\ge0$, we have
$\E[B_\varepsilon'(-\varepsilon\lambda)\eta]\le\E[g\eta]$, and
$Du_\varepsilon\cdot D\eta=-|D\eta|^2$ almost everywhere.  Hence
\[
 \E|D\eta|^2
 +\E\bigl[(B_\varepsilon'(-\varepsilon\lambda)-B_\varepsilon'(u_\varepsilon))\,\eta\bigr]
 \le0\, .
\]
On $\{\eta>0\}$ one has $u_\varepsilon<-\varepsilon\lambda<0$, so
$B_\varepsilon'(-\varepsilon\lambda)-B_\varepsilon'(u_\varepsilon)=\eta/\varepsilon$
there, and the inequality reads $\E|D\eta|^2+\varepsilon^{-1}\,\E\eta^2\le0$.  Hence $\eta=0$ and
\begin{equation}\label{eq:lower-barrier}
 u_\varepsilon\ge-\varepsilon\lambda\quad\text{a.e.}
\end{equation}

The bounded part of the decomposition will be the limit of
\[
 \mu_\varepsilon\deq\lambda+B_\varepsilon'(u_\varepsilon)
 =\lambda-\frac{u_\varepsilon^-}{\varepsilon}\, .
\]
Since $u_\varepsilon^-\ge0$, and $u_\varepsilon^-\le\varepsilon\lambda$ by
\eqref{eq:lower-barrier},
\begin{equation}\label{eq:mu-epsilon-bounds}
 0\le\mu_\varepsilon\le\lambda\, .
\end{equation}
Equation \eqref{eq:penalized-equation} becomes
\begin{equation}\label{eq:penalized-N}
 \E\ip{Du_\varepsilon}{D\varphi}
 =\E[(f-\mu_\varepsilon)\varphi]\, .
\end{equation}
Thus $u_\varepsilon\in\Dom(\cN)$ and
$\cN u_\varepsilon=f-\mu_\varepsilon$.  Taking $\varphi=1$ gives
\begin{equation}\label{eq:mu-epsilon-mass}
 \E\mu_\varepsilon=\E f\, .
\end{equation}

\emph{Uniform bounds.}
Since
$J_\varepsilon(u_\varepsilon)\le J_\varepsilon(0)=0$, dropping the non-negative
penalty term and applying \eqref{eq:penalized-source-bound} with
$v=u_\varepsilon$ gives
\[
 0\ge J_\varepsilon(u_\varepsilon)
 \ge\frac12\,\norm{Du_\varepsilon}_2^2
      -b\,\norm{Du_\varepsilon}_2+a\,\E u_\varepsilon\, .
\]
By \eqref{eq:lower-barrier}, $\E u_\varepsilon\ge-\varepsilon\lambda\ge-\lambda$, so
$\norm{Du_\varepsilon}_2$ is bounded uniformly in $\varepsilon\in(0,1)$.  Since
$\frac12x^2-bx\ge-b^2/2$ for all $x$, the preceding display gives
$\E u_\varepsilon\le b^2/(2a)$.  Together with the Poincar\'e inequality, this
proves a uniform $W^{1,2}$ bound on $u_\varepsilon$.

\emph{Passage to the limit.}
By \cref{lem:compact} in \cref{app:bochner-compactness}, after passing to a
sequence $\varepsilon\downarrow0$,
\[
 u_\varepsilon\to u\quad\text{strongly in }L^2\, ,
 \qquad
 u_\varepsilon\rightharpoonup u\quad\text{weakly in }W^{1,2}\, .
\]
By \eqref{eq:mu-epsilon-bounds}, after a further subsequence,
$\mu_\varepsilon\rightharpoonup\mu$ weakly in $L^2$.  The set
$\{z\in L^2:0\le z\le\lambda\text{ a.e.}\}$ is convex and norm-closed, hence weakly closed;
therefore $0\le\mu\le\lambda$.  Passing to the limit in
\eqref{eq:penalized-N} gives
\[
 \E\ip{Du}{D\varphi}=\E[(f-\mu)\varphi]
 \, .
\]
Thus $u\in\Dom(\cN)$ and $\cN u=f-\mu$.  Moreover,
$\norm{u_\varepsilon^-}_2\le\varepsilon\lambda$ by \eqref{eq:lower-barrier}; strong $L^2$
convergence therefore gives $u^-=0$, hence $u\ge0$.  Weak convergence also
passes \eqref{eq:mu-epsilon-mass} to the limit, so $\E\mu=\E f$.

\emph{The identity $(\lambda-\mu)u=0$.}
It remains to prove \eqref{eq:obstacle-complementarity}, that is,
$\mu=\lambda$ wherever $u>0$.  By the
definition of $\mu_\varepsilon$ and \eqref{eq:mu-epsilon-bounds},
\[
 (\lambda-\mu_\varepsilon)u_\varepsilon
 =-\varepsilon(\lambda-\mu_\varepsilon)^2\, ,
 \qquad
 \abs{(\lambda-\mu_\varepsilon)u_\varepsilon}\le\varepsilon\lambda^2\, .
\]
For every $\varphi\in L^\infty(\gamma_d)$,
\begin{align*}
 \bigl\lvert \E[\varphi \cdot (\lambda-\mu_\varepsilon)(u_\varepsilon-u)] \bigr\rvert
 &\le\lambda\,\norm{\varphi}_\infty\,\norm{u_\varepsilon-u}_1\longrightarrow0\, ,\\
 \E[\varphi \cdot (\mu_\varepsilon-\mu)u]&\longrightarrow0\, ,
\end{align*}
by strong $L^2$ convergence in the first line and weak $L^2$ convergence in the second, since
$\varphi u\in L^2$.  Hence
$\E[\varphi(\lambda-\mu)u]=\lim_{\varepsilon\downarrow 0}\E[\varphi(\lambda-\mu_\varepsilon)u_\varepsilon]=0$
for every bounded $\varphi$, which proves \eqref{eq:obstacle-complementarity}.

On $\Omega\deq\{u>0\}$, \eqref{eq:obstacle-complementarity} gives $\mu=\lambda$.
Since $\mu\ge0$, $\lambda\gamma_d(\Omega)=\E[\mu\1_\Omega]\le\E\mu=\E f$.
Since $u\ge0$, the complement is the zero set $\Omega^{\mathrm c}=\{u=0\}$, and
\cref{lem:bochner,lem:zero-set} give $D^2u=0$ almost everywhere there.  This proves \eqref{eq:obstacle-support}.
\end{proof}

\subsection{Weak-type \texorpdfstring{$(1,1)$}{(1,1)} and \texorpdfstring{$L^q$}{Lq} bounds for
  \texorpdfstring{$R_2$}{R2}}
\label{sec:weak-endpoint}

We prove a weak-type $(1,1)$ estimate for the full Hilbert--Schmidt norm of
$R_2f$, with a constant independent of $d$, and deduce its $L^q$ bounds for
$1<q\le2$.

\subsubsection*{The weak-type \texorpdfstring{$(1,1)$}{(1,1)} bound}

\begin{lemma}[$L^2$ contraction]\label{lem:R2-L2}
For every $f\in L^2(\gamma_d)$,
\[
 \norm{R_2f}_{L^2(\gamma_d;\HS_d)}
 \le\norm{f-\E f}_{L^2(\gamma_d)}
 \le\norm{f}_{L^2(\gamma_d)}\, .
\]
\end{lemma}

\begin{proof}
Let $v\deq\cN^{-1}(f-\E f)$.  Then $\cN v=f-\E f$.  By the Bochner identity (\cref{lem:bochner}),
\[
 \norm{R_2f}_2^2=\norm{D^2v}_2^2
 =\norm{\cN v}_2^2-\norm{Dv}_2^2
 \le\norm{f-\E f}_2^2\le\norm f_2^2\, . \qedhere
\]
\end{proof}

\begin{theorem}[Dimension-free weak-type $(1,1)$ bound for $R_2$]\label{thm:R2-weak11}
For every $d\ge1$, every $f\in L^1(\gamma_d)\cap L^2(\gamma_d)$, and every $t>0$,
\begin{equation}\label{eq:R2-weak11}
 \gamma_d\bigl\{|R_2f|_{\HS}>t\bigr\}
 \le\frac{2\,\norm f_{L^1(\gamma_d)}}{t}\, .
\end{equation}
The constant is independent of $d$.
The operator extends uniquely from $L^1(\gamma_d)\cap L^2(\gamma_d)$ to
$L^1(\gamma_d)$ by convergence in measure, and the same estimate holds for
the extension.
\end{theorem}

Weak-type $(1,1)$ bounds for second-order Gaussian Riesz transforms were
known in fixed dimension for each entry $D_iD_j\cN^{-1}$ separately, with
constants depending on $d$
\cite{GarciaCuervaMauceriSjogrenTorrea1999,CasarinoCiattiSjogren2021}.
\Cref{thm:R2-weak11} improves this in two respects: it controls the full
Hilbert--Schmidt norm of the Hessian, and its constant is universal.

\begin{proof}
First, let $f\in L^\infty$.  If $t\le\norm f_1$, the claimed bound is trivial.  Otherwise, apply \cref{prop:obstacle} separately to $f^+$ and $f^-$ with $\lambda=t$.  This gives
\[
 f^\pm=\cN u_\pm+\mu_\pm\, ,\qquad
 0\le\mu_\pm\le t\, ,\qquad \E\mu_\pm=\E f^\pm\, .
\]
Put $\Omega\deq\{u_->0\}\cup\{u_+>0\}$.  By \eqref{eq:obstacle-support}, $\gamma_d(\Omega)\le\norm f_1/t$, and both $D^2u_+$ and $D^2u_-$ vanish almost everywhere on $\Omega^{\mathrm c}$.  Since $\cN^{-1}\cN u_\pm=u_\pm-\E u_\pm$ by \eqref{eq:N-inverse-N}, we have $R_2f=R_2(\mu_+-\mu_-)+D^2(u_+-u_-)$, so outside $\Omega$ only the first term remains.  Applying Chebyshev's inequality to the difference of the bounded parts, gives
\begin{align*}
 \gamma_d\{|R_2f|_{\HS}>t\}
 &\le\gamma_d(\Omega)
       +\frac1{t^2}\,\norm{R_2(\mu_+-\mu_-)}_2^2\\
 &\le\frac{\norm f_1}{t}
       +\frac1{t^2}\,\E[(\mu_+-\mu_-)^2]
  \le\frac{2\,\norm f_1}{t}\, .
\end{align*}
The second line uses the $L^2$ contraction of \cref{lem:R2-L2} and the pointwise bound $(\mu_+-\mu_-)^2\le t\,(\mu_-+\mu_+)$, with expectation $t\,\norm f_1$.

Finally, for $f\in L^1\cap L^2$, choose bounded $f_n\to f$ in both $L^1$ and $L^2$.
By \cref{lem:R2-L2}, $R_2f_n\to R_2f$ in $L^2$, hence in measure.  Pass to
an almost-everywhere convergent subsequence.  For every $0<t'<t$,
\[
 \{|R_2f|_{\HS}>t\}
 \subset\liminf_{n\to\infty}
 { \{|R_2f_n|_{\HS}>t'\}}
 \quad\text{up to a null set}\, .
\]
Fatou's lemma therefore gives
$\gamma_d\{|R_2f|_{\HS}>t\}\le2\,\norm f_1/t'$.  Letting $t'\uparrow t$
proves \eqref{eq:R2-weak11}.
For general $f\in L^1(\gamma_d)$, take $f_n\in L^1(\gamma_d)\cap L^2(\gamma_d)$ with $f_n\to f$ in $L^1$.  Applying \eqref{eq:R2-weak11} to $f_m-f_n$ shows that $(R_2f_n)_{n\ge 0}$ is Cauchy in measure.  Its limit is independent of the approximating sequence and satisfies the same weak-type bound by the preceding limiting argument.
\end{proof}

\subsubsection*{Interpolation to \texorpdfstring{$L^q$}{Lq}}

We record a standard interpolation estimate with an explicit constant.

\begin{lemma}[Interpolation of weak-type $(1,1)$ and $L^2$ bounds]\label{lem:marcinkiewicz}
Let $\nu$ be a finite measure, and let $T$ be a linear operator, initially defined on
$L^1(\nu)\cap L^2(\nu)$,
from scalar functions to functions with values in a finite-dimensional
Hilbert space $\cH$.  Suppose
\[
 \nu\{|Tf|_\cH>t\}\le A\,\norm f_{L^1(\nu)}/t\, ,
 \qquad
 \norm{Tf}_{L^2(\nu;\cH)}\le B\,\norm f_{L^2(\nu)}\, .
\]
Then, for every $1<q\le2$,
\begin{equation*}
    \norm{Tf}_{L^q(\nu;\cH)}\leq\frac{8\,(A+B)}{q-1}\,\norm f_{L^q(\nu)}\, ,
\end{equation*}
initially for $f\in L^q(\nu)\cap L^2(\nu)$ and then, by density, on all of
$L^q(\nu)$.
\end{lemma}

\begin{proof}
    The interpolation argument is standard; see~\cite[Exercise 1.3.2]{Gra14Fourier}, which gives the constant $8$. That exercise is for a scalar-valued map $T$, but for each $\cH$-valued function $h$ with $\abs h_{\cH} \le 1$, we can apply it to the map $T_h : f \mapsto \langle Tf, h \rangle_{\cH}$, and then for non-zero $f$, take $h = Tf/\abs{Tf}_{\cH}$.
\end{proof}

\begin{corollary}[$L^q$ bound of order $(q-1)^{-1}$ for $R_2$]\label{cor:R2-Lq}
For every $1<q\le2$ and every $f\in L^q(\gamma_d)$,
\begin{equation}\label{eq:R2-Lq}
 \norm{R_2f}_{L^q(\gamma_d;\HS_d)}
 \le\frac{24}{q-1}\,\norm f_{L^q(\gamma_d)}\, .
\end{equation}
\end{corollary}

\begin{proof}
Apply \cref{lem:marcinkiewicz} to $T=R_2$ with $A=2$ and $B=1$, using
\cref{thm:R2-weak11,lem:R2-L2}. The lemma also gives the unique bounded
extension of $R_2$ to $L^q(\gamma_d)$.
\end{proof}

The order $(q-1)^{-1}$ in \eqref{eq:R2-Lq} cannot be improved: the lower
bound in \cref{cor:even-Riesz} with $r=1$ shows that the norm of $R_2$ on
$L^q(\gamma_d)$ is at least $c/(q-1)$.

\subsection{The shifted transform and its Hilbert-valued extension}
\label{sec:shift}

We prove the $L^q$ bound for the shifted second-order transform
$(\cN+1)^{-1}D^2$ and extend it to Hilbert-valued inputs, with constants
independent of the base and target dimensions.

\begin{lemma}[Shifted-transform identity]\label{lem:shifted-identity}
For every $F\in C_{\rm c}^\infty(\R^d)$,
\begin{equation}\label{eq:shifted-identity}
 (\cN+1)^{-1}D^2F
 =\bigl[\Id+(\cN+1)^{-1}\bigr]R_2F\, .
\end{equation}
Consequently, for $1<q\le2$,
\begin{equation}\label{eq:shifted-Lq-scalar}
 \norm{(\cN+1)^{-1}D^2F}_{L^q(\gamma_d;\HS_d)}
 \le\frac{48}{q-1}\,\norm F_{L^q(\gamma_d)}\, .
\end{equation}
\end{lemma}

\begin{proof}
First let $F_n$ lie in the $n$-th Wiener chaos, with $n\ge2$.
Since $D^2F_n$ lies in the $(n-2)$-nd Wiener chaos,
\[
 (\cN+1)^{-1}D^2F_n=\frac1{n-1}\,D^2F_n\, ,
 \qquad
 R_2F_n=\frac1n\,D^2F_n\, .
\]
Applying $\Id+(\cN+1)^{-1}$ to the latter multiplies it by
$1+1/(n-1)=n/(n-1)$, proving the identity on finite chaos sums.  Both sides vanish on the
zeroth and first chaoses.

For a general $F\in C_{\rm c}^\infty(\R^d)$, let $F^{(N)}$ be its chaos truncations.  Since
$F\in\Dom(\cN)$, \cref{lem:bochner}, applied to $F-F^{(N)}$, gives
$D^2F^{(N)}\to D^2F$ in $L^2(\gamma_d;\HS_d)$.  The $L^2$ contraction in
\cref{lem:R2-L2} gives $R_2F^{(N)}\to R_2F$ in $L^2$, and the resolvent
$(\cN+1)^{-1}$ is an $L^2$ contraction.  Passing to the limit proves
\eqref{eq:shifted-identity} in $L^2$ for $F$.

By \cref{lem:OU}(v), $(\cN+1)^{-1}$ is a contraction on
$L^q(\gamma_d;\HS_d)$.  Combining
\eqref{eq:shifted-identity} with \cref{cor:R2-Lq} proves \eqref{eq:shifted-Lq-scalar}, since $\Id+(\cN+1)^{-1}$ has norm at most $2$.
\end{proof}

\begin{lemma}[Dimension-free Hilbert-valued extension]\label{lem:Hilbert-extension}
Let $\cH$ be a finite-dimensional real Hilbert space and $1<q\le2$.  For every
$F\in C_{\rm c}^\infty(\R^d;\cH)$,
\begin{equation}\label{eq:shifted-Lq-H}
 \norm{(\cN+1)^{-1}D^2F}_{L^q(\gamma_d;\cH^{d\times d})}
 \le\frac{48}{q-1}\,\norm F_{L^q(\gamma_d;\cH)}\, ,
\end{equation}
\end{lemma}

\begin{proof}
    We use the standard Marcinkiewicz--Zygmund extension argument.
    Choose an orthonormal basis $(e_a)_{a\ge 0}$ of $\cH$ and write $F=\sum_{a\ge 0} F_ae_a$.  Let
    $(g_a)_{a\ge 0}$ be independent standard Gaussians, indexed by the same basis and
    independent of $x$, and consider the Gaussian randomization $F_g\deq\sum_{a\ge 0} g_aF_a$.  Denote $T\deq(\cN+1)^{-1}D^2$, which is linear and does not act on $g$.
By the lower bound in \cref{lem:gaussian-randomization}, applied pointwise
in the Hilbert space $\HS_d$,
\[
 |TF(x)|_{\cH^{d\times d}}^q
 \le m_q^{-q}\,\E_g|T F_g(x)|_{\HS_d}^q\, .
\]
For the scalar randomization, \cref{lem:Gaussian-linear-forms}(i) gives
$\E_g\abs{F_g(x)}^q=m_q^q\,\norm{F(x)}_\cH^q$.
Integrating in $x$ and applying \eqref{eq:shifted-Lq-scalar}, we therefore obtain
\begin{align*}
 \norm{TF}_{L^q(\cH^{d\times d})}^q
 &\le m_q^{-q}\,\E_g\norm{TF_g}_{L^q(\HS_d)}^q
 \le \Bigl(\frac{48}{m_q\,(q-1)}\Bigr)^q\,\E_g\norm{F_g}_{L^q}^q
 =\Bigl(\frac{48}{q-1}\Bigr)^q\,\norm F_{L^q(\cH)}^q\, .
\end{align*}
Thus the factors $m_q$ cancel, and \eqref{eq:shifted-Lq-H} holds with the
same constant as the scalar estimate, independently of $\dim \cH$.
\end{proof}

\subsection{Duality and proof of the divergence inequality}
\label{sec:duality}

We prove the divergence inequality for Hilbert-valued fields, completing
the upper bounds in \cref{thm:main,cor:Hilbert-main}.  The following identity
is an integration by parts; we record the operator domains that justify it.

\begin{lemma}[Centered duality]\label{lem:centered-duality}
Let $\cH$ be a finite-dimensional real Hilbert space and let
$V\deq (V_i)_{i=1}^d\in W^{1,2}(\gamma_d;\cH^d)$ satisfy $\E V_i=0$ for every $i$.  Then, for
every $F\in C_{\rm c}^\infty(\R^d;\cH)$,
\begin{equation*}
 \ip{\delta V}{F}
 =\E\sum_{i,j=1}^d
 \ip{(\cN+1)^{-1}D_jD_iF}{D_jV_i}_\cH\, .
\end{equation*}
\end{lemma}

\begin{proof}
For each $i$, centering gives $u_i\deq\cN^{-1}V_i\in\Dom(\cN)$
with $\cN u_i=V_i$.  By \cref{lem:N-weak}(i)--(ii),
$Du_i=(\cN+1)^{-1}DV_i$ lies in $\Dom(\delta)$ and
$\delta Du_i=V_i$.  Since $D_iF\in W^{1,2}$, the adjoint relation and
self-adjointness of $(\cN+1)^{-1}$ on $L^2$ give
\begin{align*}
 \ip{\delta V}{F}
 &=\sum_{i=1}^d\E\ip{D_iF}{V_i}_\cH \\
 &=\sum_{i,j=1}^d\E\ip{D_jD_iF}{(\cN+1)^{-1}D_jV_i}_\cH
 =\sum_{i,j=1}^d\E\ip{(\cN+1)^{-1}D_jD_iF}{D_jV_i}_\cH\, .
\end{align*}
All terms are finite because $F$ is smooth and compactly supported and $V\in W^{1,2}$.
\end{proof}

\begin{proof}[Proof of \cref{cor:Hilbert-main} and the upper bound in \cref{thm:main}]
Write $V_0\deq V-\E V$.  Since $p\ge2$, one has
$V_0\in W^{1,2}(\gamma_d;\cH^d)$.  Recall that $p'=p/(p-1)\in(1,2]$.  For
$F\in C_{\rm c}^\infty(\R^d;\cH)$, \cref{lem:centered-duality,lem:Hilbert-extension} and
H\"older's inequality give
\begin{align*}
 \abs{\ip{\delta V_0}{F}}
 &\le
 \norm{(\cN+1)^{-1}D^2F}_{L^{p'}(\cH^{d\times d})}
 \,\norm{DV}_{L^p(\cH^{d\times d})}\\
 &\le \frac{48}{p'-1}\,\norm F_{L^{p'}(\cH)}\,\norm{DV}_{L^p(\cH^{d\times d})}\, .
\end{align*}
The test class $C_{\rm c}^\infty(\R^d;\cH)$ is dense in $L^{p'}(\gamma_d;\cH)$.  Hence the distribution
$\delta V_0$ extends to a bounded linear functional on that space.  Since $\cH$ is
finite-dimensional, $L^p(\gamma_d;\cH)$ is the dual of $L^{p'}(\gamma_d;\cH)$, so this
functional is given by an element of $L^p(\gamma_d;\cH)$, and since
$(p'-1)^{-1}=p-1\le p$,
\begin{equation}\label{eq:centered-final}
 \norm{\delta V_0}_{L^p(\cH)}\le 48\,(p-1)\,\norm{DV}_{L^p(\cH^{d\times d})}\, .
\end{equation}

For the constant field $\E V$, one has
$\delta(\E V)=\sum_iX_i\,\E V_i$ with $X\sim\gamma_d$.
\Cref{lem:gaussian-randomization} gives
\[
 \norm{\delta(\E V)}_{L^p(\cH)}
 \leq m_p\,\biggl(\sum_{i=1}^d\norm{\E V_i}_\cH^2\biggr)^{1/2}
 \leq\sqrt p\,\norm{\E V}_{\cH^d}\, .
\]
Combining this with \eqref{eq:centered-final} proves
\eqref{eq:Hilbert-main}; the scalar target $\R$ gives \eqref{eq:main}.
\end{proof}

\begin{remark}
    In~\cite[Lemma C.3]{Chen+26PicardHMC}, we showed that $(\cN+1)^{-1}D^2$ can be written as a bounded operator composed with two first-order Riesz transforms. By applying the bound for the first-order Riesz transform twice, we obtained \cref{thm:main} with the suboptimal constant $Cp^2$ rather than $Cp$ in front of the second term. The improvement here comes from directly treating the second-order Riesz transform via the Calder\'on--Zygmund decomposition.
\end{remark}

\subsection{Even-order Riesz transforms}
\label{sec:even-Riesz}

We prove the bounds for the higher even-order transforms and their
adjoints, including the matching lower bound in \cref{cor:even-Riesz}.

\begin{proof}[Proof of \cref{cor:even-Riesz}]
\emph{Upper bounds.}
For $a\geq0$, put
\[
 S_a\deq D^2(\cN+a)^{-1}\, ,
\]
where $S_0=D^2\cN^{-1}=R_2$.  The scalar estimate in
\cref{cor:R2-Lq}, followed by the same Gaussian randomization as in
\cref{lem:Hilbert-extension}, gives
\begin{equation}
 \norm{S_0F}_{L^q(\gamma_d;\cH^{d\times d})}
 \leq\frac{24}{q-1}\,\norm F_{L^q(\gamma_d;\cH)}
 \label{eq:R2-Hilbert}
\end{equation}
If $a>0$, then $S_a=S_0M_a$ with $M_a=\cN(\cN+a)^{-1}$ as in
\eqref{eq:Ma}, and
\cref{lem:OU}(v) gives $\norm{M_a}_{L^q(\cH)\to L^q(\cH)}\leq2$.
Consequently, every $S_a$ satisfies
\[
 \norm{S_aF}_{L^q(\gamma_d;\cH^{d\times d})}
 \leq\frac{48}{q-1}\,\norm F_{L^q(\gamma_d;\cH)}\, .
\]

To see how the shifts combine, let $f_n$ belong to the $n$-th Wiener chaos
with $n\ge2r$.  After $k$ pairs of derivatives, the chaos order is $n-2k$,
so
\[
 S_{2k}(D^{2k}f_n)=\frac1n\,D^{2k+2}f_n\, ,
 \qquad k=0,\dots,r-1\, .
\]
Each step therefore contributes a factor $1/n$.  This proves, first on
finite chaos sums,
\begin{equation*}
 D^{2r}\cN^{-r}
 =S_{2r-2}\circ S_{2r-4}\circ\cdots\circ S_2\circ S_0\, .
\end{equation*}
Both sides vanish on chaoses of order less than $2r$.  Iterating
\eqref{eq:R2-Hilbert} for $S_0$ and the bound $48/(q-1)$ for each subsequent factor, then using density proves
\[
 \norm{D^{2r}\cN^{-r}f}_{L^q((\R^d)^{\otimes 2r})}
 \leq\frac{24\cdot48^{r-1}}{(q-1)^r}\,\norm f_{L^q}\, .
\]
The same argument, beginning with an arbitrary finite-dimensional $\cH$, proves
the asserted Hilbert-valued extension.  Let $p\ge2$ be the exponent with
$p'=q$, so that $q-1=(p-1)^{-1}$.  Taking Banach space adjoints proves
\eqref{eq:even-Riesz-adjoint}, first on the stated test tensors and then, by
density, as a bounded operator on the full tensor-valued $L^p$ space.

\emph{Lower bound.}  On smooth finite
chaos tensors, the adjoint of the full transform is
\[
 (D^{2r}\cN^{-r})^*=\cN^{-r}\delta^{2r}\, .
\]
In dimension one, apply this operator to the constant tensor $1$.  By
\eqref{eq:Hermite-D-delta} and \eqref{eq:chaos-shifts},
\[
 \delta^{2r}1=H_{2r}\, ,
 \qquad
 \cN^{-r}H_{2r}=(2r)^{-r}H_{2r}\, .
\]
Since $H_{2r}(x)=x^{2r}$ plus terms of lower degree,
\cref{lem:Gaussian-moments} and the reverse triangle inequality give
\[
 \norm{H_{2r}}_{L^p(\gamma_1)}\geq c_rp^r\, ,
 \qquad p\geq2\, .
\]
Indeed, the leading term dominates for all sufficiently large $p$, and the
remaining compact interval is absorbed by decreasing $c_r$.  Therefore
\[
 \norm{D^{2r}\cN^{-r}}_{L^q\to L^q}
 =\norm{\cN^{-r}\delta^{2r}}_{L^p\to L^p}
 \geq c_rp^r
 \geq\frac{c_r}{(q-1)^r}\, .
\]
The one-dimensional example embeds in every $d\geq1$, completing the proof.
\end{proof}

\section{Proof of the Bernstein--Markov inequality and its consequences}
\label{sec:proof-Bernstein-Markov}

This section proves \cref{thm:Bernstein-Markov} and its consequence for
polynomial vector fields, \cref{cor:polynomial-divergence}.
After introducing the complex Gaussian measures and the Fock space in
\cref{sec:Fock}, we proceed in four steps.
\begin{enumerate}
\item \emph{Moment recursion.}  \Cref{sec:moment-recursion} uses Gaussian
integration by parts in the Fock space to prove a recursion for moments of
fractional powers of a homogeneous holomorphic polynomial.

\item \emph{Homogeneous comparison.}  \Cref{sec:homogeneous-comparison}
expresses an elliptic complex Gaussian moment as a series in these moments.
The recursion gives monotonicity and hence a comparison with the real
Gaussian moment.

\item \emph{General polynomials and complex rotations.}
\Cref{sec:complex-rotations} homogenizes the even and odd parts of a
polynomial separately, by multiplying their lower degree components with
powers of a normalized sum of squares of auxiliary Gaussian variables, and
obtains a bound for the complex rotations of the polynomial.

\item \emph{Gradient estimate.}  \Cref{sec:cauchy} applies Cauchy's formula
to the complex rotation, bounding the Gaussian directional derivative by
$C\sqrt n$ times the polynomial norm.  Conditioning on the base point
converts this derivative into $m_p$ times the gradient norm, giving the
factor $\sqrt{n/p}$.  We then deduce the estimate for polynomial vector fields.
\end{enumerate}

\subsection{Complex Gaussian measures and the Fock space}
\label{sec:Fock}

The proof of \cref{thm:Bernstein-Markov} evaluates polynomials at complex
Gaussian arguments.  We identify $\mathbb C^M$ with $\R^{2M}$, write
$z=x+iy$, and use the complex-bilinear dot product $z\cdot w\deq\sum_jz_jw_j$.
In this subsection and in the proof of \cref{lem:Fock-recursion}, we use $\nabla$ to
denote the full real gradient on $\R^{2M}$ and $\nabla^a$ the tensor of
real derivatives of order $a$, to distinguish them from the Gaussian gradient $D$ on $\R^d$.
We will also use two quadratic forms:
\[
 \abs z^2\deq z\cdot\bar z=\sum_{j=1}^M\abs{z_j}^2\, ,
 \qquad
 Q(z)\deq z\cdot z=\sum_{j=1}^Mz_j^2\, .
\]
The holomorphic coordinate derivatives are
$\partial_{z_j}\deq\frac12(\partial_{x_j}-i\partial_{y_j})$.
Let $E$ denote the holomorphic Euler operator and $\Delta_z$ the
holomorphic Laplacian:
\[
 E\deq\sum_{j=1}^Mz_j\partial_{z_j}\, ,
 \qquad
 \Delta_z\deq\sum_{j=1}^M\partial_{z_j}^2\, .
\]
We also write $Q$ for the operator of multiplication by the function
$Q(z)$, so that $Q^kh=Q(z)^kh(z)$; \cref{lem:Fock-IBP} shows that
$Q$ and $\Delta_z$ are adjoint on holomorphic polynomials with respect to
the Fock pairing defined below.
By Euler's theorem, a
holomorphic function $h$ on an open cone satisfies $Eh=\beta h$ if and only
if $h(\lambda z)=\lambda^\beta h(z)$ for all $\lambda>0$; this applies in
particular to local branches of $H^q$ when $H$ is a homogeneous polynomial of
degree $n$, with $\beta=nq$, and to $Q^kH^q$ with $\beta=nq+2k$, where $q$ could be a fractional number.

\paragraph{The standard complex Gaussian.}
The standard complex Gaussian measure on $\mathbb C^M$ is
\[
 \gamma_{\mathbb C}^M(\dd z)\deq\pi^{-M}e^{-\abs z^2}\,\dd z\, ,
\]
where $\dd z$ is Lebesgue measure on $\R^{2M}$.  It is the law of
$Z\deq(X+iY)/\sqrt2$ with $X,Y$ independent standard Gaussian vectors in
$\R^M$, and it is characterized among centered Gaussian laws on $\R^{2M}$ by
\[
 \E[Z_jZ_k]=0\, ,
 \qquad
 \E[Z_j\overline{Z_k}]=\delta_{jk}\, .
\]
It is invariant under the unitary group of $\mathbb C^M$, in particular
under the phase rotations $z\mapsto e^{i\theta}z$.  The monomials are
orthogonal,
\begin{equation}\label{eq:Fock-monomials}
 \int_{\mathbb C^M}z^\alpha\,\overline{z^\beta}\,\gamma_{\mathbb C}^M(\dd z)
 =\alpha!\,\delta_{\alpha\beta}\, ,
\end{equation}
as one sees by expressing each variable in polar coordinates.

\paragraph{The Fock space.}
The Fock space, also called the Segal--Bargmann space, is the closed
subspace of $L^2(\gamma_{\mathbb C}^M)$ consisting of entire functions
\cite{Bargmann1961,Zhu2012}; by \eqref{eq:Fock-monomials}, the normalized
monomials $z^\alpha/\sqrt{\alpha!}$ form an orthonormal basis of it.  We use
the pairing
\begin{equation*}
 \ip fg\deq\pi^{-M}\int_{\mathbb C^M}f\,\overline g\,e^{-\abs z^2}\,\dd z\, ,
\end{equation*}
which is linear in the first variable.  The only structural property of the
Fock space that we need is that multiplication by $z_j$ and differentiation
$\partial_{z_j}$ are mutually adjoint.  This is Gaussian integration by parts in
complex form.

\begin{lemma}[Adjoints in the Fock space]\label{lem:Fock-IBP}
Let $f$ and $g$ be holomorphic on $\mathbb C^M$, with $f$, $g$, $\nabla f$,
and $\nabla g$ of at most polynomial growth.  Then
\begin{equation}\label{eq:Fock-IBP}
 \ip{\partial_{z_j}f}{g}=\ip{f}{z_jg}\, ,
 \qquad
 \ip{z_jf}{g}=\ip{f}{\partial_{z_j}g}\, ,
\end{equation}
and consequently $\ip{Qf}{g}=\ip f{\Delta_z g}$.  Moreover, on
holomorphic functions,
\begin{equation}\label{eq:Fock-commutator-prelim}
 [\Delta_z,Q]
 =\Delta_zQ-Q\Delta_z
 =4E+2M\,\Id\, ,
\end{equation}
where $\Id$ is the identity operator; thus the last term maps $f$ to $2Mf$.
\end{lemma}

\begin{proof}
Since $g$ is holomorphic, $\partial_{z_j}\overline g=0$, and
$\partial_{z_j}e^{-z\cdot\bar z}=-\overline{z_j}\,e^{-\abs z^2}$.  Hence
\[
 \partial_{z_j}\bigl(f\,\overline g\,e^{-\abs z^2}\bigr)
 =(\partial_{z_j}f)\,\overline g\,e^{-\abs z^2}
  -f\,\overline{z_jg}\,e^{-\abs z^2}\, .
\]
The operator $\partial_{z_j}$ is a combination of real partial derivatives, so
the integral of the left-hand side over $\R^{2M}$ vanishes; this is
justified by a cutoff argument using the polynomial growth.
This proves the first identity in \eqref{eq:Fock-IBP},
and the second follows by conjugation and exchanging the roles of $f$ and
$g$.  Applying the second identity twice gives $\ip{Qf}g=\ip f{\Delta_z g}$.
For \eqref{eq:Fock-commutator-prelim}, compute
\[
 \partial_{z_j}^2(Qf)
 =\partial_{z_j}\bigl(2z_jf+Q\partial_{z_j}f\bigr)
 =2f+4z_j\partial_{z_j}f+Q\partial_{z_j}^2f
\]
and sum over $j$.
\end{proof}

\paragraph{Elliptic complex Gaussians.}
For $s\in\R$ and independent standard Gaussians $X,Y$ in $\R^d$, put
\begin{equation*}
 Z_s\deq X\cosh s+iY\sinh s\, .
\end{equation*}
Thus $Z_0=X$ is real, while for $s\neq0$ the vector $Z_s$ is a complex
Gaussian vector with covariance and pseudo-covariance
\begin{equation}\label{eq:Zs-covariances}
 \E\bigl[Z_s\overline{Z_s}^{\mathsf T}\bigr]=(\cosh2s)\,I\, ,
 \qquad
 \E\bigl[Z_sZ_s^{\mathsf T}\bigr]=(\cosh^2s-\sinh^2s)\,I=I\, .
\end{equation}
The pseudo-covariance is the quantity that enters the Hermite generating
function \eqref{eq:Hermite-generating-d}, and it is the same as for $X$.
Hermite polynomials therefore remain orthogonal at the argument $Z_s$, with
norms rescaled by the covariance.

\begin{lemma}[Hermite polynomials at elliptic complex Gaussian arguments]
\label{lem:complex-chaos}
For all $\alpha,\beta\in\mathbb N^d$ and $s\in\R$,
\begin{equation}\label{eq:complex-Hermite-orthogonality}
 \E\bigl[H_\alpha(Z_s)\overline{H_\beta(Z_s)}\bigr]
 =\alpha!\,(\cosh2s)^{\abs\alpha}\,\delta_{\alpha\beta}\, .
\end{equation}
Consequently, if $F=\sum_{m\ge 0} F_m$ is the chaos expansion of a polynomial $F$
on $\R^d$ with real or complex coefficients, then
\begin{equation}\label{eq:complex-chaos-identity}
 \E\abs{F(Z_s)}^2
 =\sum_{m\ge 0} (\cosh2s)^m\,\norm{F_m}_{L^2(\gamma_d;\mathbb C)}^2\, .
\end{equation}
\end{lemma}

\begin{proof}
    The formula~\eqref{eq:complex-Hermite-orthogonality} follows from the standard Wick product formula
    \cite[Theorem~3.9]{Janson1997}, after noting that the Wick monomial ${:}Z_s^\alpha{:}$ equals $H_\alpha(Z_s)$.

    Write $F=\sum_{\alpha\in\mathbb N^d} c_\alpha H_\alpha$, a finite sum.
Orthogonality gives
\[
 \E\abs{F(Z_s)}^2
 =\sum_{\alpha\in\mathbb N^d} \abs{c_\alpha}^2\alpha!\,(\cosh2s)^{\abs\alpha}
 =\sum_{m\geq0}(\cosh2s)^m
       \sum_{\abs\alpha=m}\abs{c_\alpha}^2\alpha!\, .
\]
Since $F_m=\sum_{\abs\alpha=m}c_\alpha H_\alpha$ and
$\norm{F_m}_{L^2(\gamma_d;\mathbb C)}^2
 =\sum_{\abs\alpha=m}\abs{c_\alpha}^2\alpha!$, this is
\eqref{eq:complex-chaos-identity}.
\end{proof}

\subsection{A moment recursion in the Fock space}
\label{sec:moment-recursion}

We prove a recursion for the moments
$\int\abs H^{2q}\abs Q^{2k}\,\dd\gamma_{\mathbb C}^M$ of a homogeneous
holomorphic polynomial $H:\mathbb C^M\to\mathbb C$, where the exponent
$q\geq1$ is real; in the application to \cref{thm:Bernstein-Markov} we
take $q\deq p/2$, so the assumption $p\geq2$ is exactly $q\geq1$. The calculation rests on the adjoint
relation $Q^*=\Delta_z$ of \cref{lem:Fock-IBP}, and is closely related to the Newman--Shapiro identity~\cite[Theorem 3]{NewSha1966}.  For non-integer $q$, the
function $H^q$ need not be single-valued or holomorphic across the zeros
of $H$, so we must justify the adjoint operations using local branches.

\begin{lemma}[Fock space moment recursion]
\label{lem:Fock-recursion}
Let $H:\mathbb C^M\to\mathbb C$ be a non-zero homogeneous holomorphic
polynomial of degree $n\geq1$, let $q\geq1$, and set
\[
 \alpha\deq nq+\frac M2\, ,
 \qquad
 N_k\deq\int_{\mathbb C^M}\abs{H(z)}^{2q}\abs{Q(z)}^{2k}\,
       \gamma_{\mathbb C}^M(\dd z)\, .
\]
Then, for every integer $k\geq0$,
\begin{equation}
 N_{k+1}\geq4(k+1)(\alpha+k)N_k\, .
 \label{eq:Fock-recursion}
\end{equation}
\end{lemma}

\begin{proof}
We first give the algebraic argument, then justify the adjoint operations
for fractional powers.

On $\Omega\deq\mathbb C^M\setminus\{H=0\}$, choose local branches $h\deq H^q$ and put
\[
 g_k\deq Q^kh\, .
\]
On overlaps, two branches differ by a constant of modulus one.  Their
derivatives acquire the same factor, so the pointwise norms and the
integrands of all pairings below are independent of the branch.
All norms and pairings of these local branches are defined by these
single-valued integrands.
Homogeneity gives
\begin{equation}
 \norm{g_k}_{L^2(\gamma_{\mathbb C}^M)}^2=N_k>0\, ,
 \qquad
 Eg_k=(nq+2k)g_k\, .
 \label{eq:gk-norm-Euler}
\end{equation}

\emph{The algebraic argument.}
For integer $q$, the functions $g_k$ are polynomials, so the adjoint relation
$Q^*=\Delta_z$ follows from \cref{lem:Fock-IBP}.  For non-integer $q>1$, we
first make the calculation under the same adjoint relations; the
justification and approximation follow below.  The commutator identity
\eqref{eq:Fock-commutator-prelim} and \eqref{eq:gk-norm-Euler} give
\begin{align}
 N_{k+1}
 &=\ip{Qg_k}{Qg_k}=\ip{g_k}{\Delta_z(Qg_k)} \notag\\
 &=\ip{g_k}{Q\Delta_z g_k}
   +4(\alpha+2k)N_k \notag\\
 &=\norm{\Delta_z g_k}_2^2
   +4(\alpha+2k)N_k\, .
 \label{eq:Fock-norm-identity}
\end{align}
For $k\geq1$, the same adjoint relation and $g_k=Qg_{k-1}$ give
\[
 N_k=\ip{g_{k-1}}{\Delta_z g_k}\, .
\]
Hence Cauchy--Schwarz yields
\begin{equation}
 \norm{\Delta_z g_k}_2^2
 \geq\frac{N_k^2}{N_{k-1}}\, .
 \label{eq:Fock-Cauchy-Schwarz}
\end{equation}

The base case $N_1\geq4\alpha N_0$ follows from
\eqref{eq:Fock-norm-identity}.  If, inductively,
\[
 N_k\geq4k(\alpha+k-1)N_{k-1}\, ,
\]
then \eqref{eq:Fock-Cauchy-Schwarz} and
\eqref{eq:Fock-norm-identity} imply
\begin{align*}
 N_{k+1}
 &\geq4\bigl\{k(\alpha+k-1)+\alpha+2k\bigr\}N_k
 =4(k+1)(\alpha+k)N_k\, .
\end{align*}
This proves the recursion once the adjoint operations are justified.

\emph{Justification for fractional powers.}
Suppose $q>1$ and $M\geq2$, and first assume that the derivatives
$\partial_{z_1}H,\ldots,\partial_{z_M}H$ do not vanish simultaneously at
any point $z\ne0$ with $H(z)=0$.  Near each such point, we can use $w=H(z)$
as one of the local coordinates, and second derivatives of $g_k$ are
$O(\abs w^{q-2})$.
Thus $q>1$ ensures local square-integrability and vanishing boundary terms
in integration by parts.  Homogeneity controls the origin and Gaussian
decay controls infinity, so the adjoint relations above hold.  For general
$H$, approximate its coefficients by homogeneous polynomials of the same
degree satisfying this derivative condition and pass
\eqref{eq:Fock-recursion} to the limit by dominated convergence.

Finally, if $M=1$, then $H(z)=cz^n$ and polar integration gives
$N_k=\abs c^{2q}\Gamma(nq+2k+1)$.  Consequently,
\begin{align*}
 \frac{N_{k+1}}{N_k}
 &=(nq+2k+1)(nq+2k+2)\\
 &=4(k+1)(\alpha+k)+nq(nq-1)
 \geq4(k+1)(\alpha+k)\, ,
\end{align*}
which proves the recursion in this case as well.
\end{proof}

\subsection{The homogeneous complex Gaussian comparison}
\label{sec:homogeneous-comparison}

The moment recursion gives a comparison between the $p$-th moment of a
homogeneous polynomial at an elliptic complex Gaussian and its moment at a
real Gaussian.

\begin{proposition}[Homogeneous complex Gaussian comparison]
\label{prop:homogeneous-complex-comparison}
Let $H:\mathbb C^M\to\mathbb C$ be a homogeneous holomorphic polynomial of
degree $n\geq0$, and let
$X,Y$ be independent standard real Gaussian vectors in $\R^M$.  For every
$p\geq2$ and $s\in\R$, with $Z_s=X\cosh s+iY\sinh s$,
\begin{equation}
 \norm{H(Z_s)}_{L^p(\gamma_M\otimes\gamma_M;\mathbb C)}
 \leq(\cosh 2s)^{n/2}\,\norm{H(X)}_{L^p(\gamma_M;\mathbb C)}\, .
 \label{eq:homogeneous-complex-comparison}
\end{equation}
\end{proposition}

\begin{remark}
    When $p$ is restricted to be an even integer, then the inequality follows directly from \cref{lem:complex-chaos}, without requiring homogeneity of $H$. Thus, the difficult part is to extend the bound to all real $p\ge 2$.
\end{remark}

\begin{proof}
The assertion is immediate for $H=0$ or $n=0$, so suppose that $H\neq0$ and
$n\geq1$.  Use the moments $N_k$ of \cref{lem:Fock-recursion} with
$q\deq p/2$ and, for $0\leq t\leq1$, define
\[
 W_t\deq\sqrt{\frac{1+t}{2}}\,X
      +i\sqrt{\frac{1-t}{2}}\,Y\, ,
 \qquad
 F(t)\deq\E\abs{H(W_t)}^p\, .
\]
Thus, $W_0$ has law $\gamma_{\mathbb C}^M$ and $W_1=X$.  We will show that
$F$ is non-decreasing, so the real Gaussian gives the largest moment in
this family.

For $0\leq t<1$, write a point of $\mathbb C^M$ as $z=x+iy$, with
$x,y\in\R^M$.  The density of the law of $W_t$, evaluated at $z$ with
respect to Lebesgue measure $\dd x\,\dd y$, is
\[
 \pi^{-M}(1-t^2)^{-M/2}
 \exp\Bigl\{-\frac{\abs x^2}{1+t}-\frac{\abs y^2}{1-t}\Bigr\}\, .
\]
Since
\[
 \frac{\abs x^2}{1+t}+\frac{\abs y^2}{1-t}
 =\frac{\abs z^2-t\operatorname{Re}Q(z)}{1-t^2}\, ,
\]
the change of variables $z=\sqrt{1-t^2}\,w$ and homogeneity give, for
$0\leq t<1$,
\begin{equation}
 F(t)=(1-t^2)^\alpha
 \int_{\mathbb C^M}\abs{H(w)}^p\,
 e^{t\operatorname{Re}Q(w)}\,\gamma_{\mathbb C}^M(\dd w)\, ,
 \qquad
 \alpha=\frac{np+M}{2}\, .
 \label{eq:elliptic-density-identity}
\end{equation}
The exponent $\alpha=nq+M/2$ is the same as in the moment recursion:
homogeneity contributes $np/2$, while the density and Jacobian together
contribute $M/2$.

Global phase invariance gives
$\abs{H(e^{i\theta}w)}=\abs{H(w)}$ and
$Q(e^{i\theta}w)=e^{2i\theta}Q(w)$.  Averaging $e^{r\cos\theta}$ over the
circle produces the modified Bessel function of order zero,
\begin{equation}\label{eq:Bessel-I0}
 \frac1{2\pi}\int_0^{2\pi}e^{r\cos\theta}\,\dd\theta
 =\sum_{k\geq0}\frac{r^{2k}}{4^k\,(k!)^2}\, ,
 \qquad r\in\R\, ,
\end{equation}
because odd powers of $\cos\theta$ average to zero and
$\frac1{2\pi}\int_0^{2\pi}\cos^{2k}\theta\,\dd\theta=\binom{2k}{k}4^{-k}$.
Averaging \eqref{eq:elliptic-density-identity} over $\theta$, applying
\eqref{eq:Bessel-I0} with $r=t\abs{Q(w)}$, and integrating term by term,
which Tonelli's theorem allows since all terms are non-negative, yields
\begin{equation}
 F(t)=(1-t^2)^\alpha
 \sum_{k=0}^\infty\frac{t^{2k}}{4^k\,(k!)^2}\,N_k\, .
 \label{eq:Bessel-expansion}
\end{equation}

Set
\[
 b_k\deq\frac{N_k}{4^k\,(k!)^2}\, ,
 \qquad
 A(u)\deq\sum_{k=0}^\infty b_ku^k\, ,
\]
so that \eqref{eq:Bessel-expansion} reads $F(t)=(1-t^2)^\alpha A(t^2)$.
The coefficients $b_k$ are non-negative and $F(t)<\infty$ for $0\leq t<1$,
so the power series $A$ converges on $[0,1)$.  The recursion
\eqref{eq:Fock-recursion} divided by $4^{k+1}((k+1)!)^2$ reads
$(k+1)b_{k+1}\geq(\alpha+k)b_k$, which says that every coefficient of
\[
 (1-u)A'(u)-\alpha A(u)
 =\sum_{k=0}^\infty\bigl\{(k+1)b_{k+1}-(\alpha+k)b_k\bigr\}u^k
\]
is non-negative.  Hence, for $0\leq u<1$,
\[
 \frac{\dd}{\dd u}\bigl\{(1-u)^\alpha A(u)\bigr\}
 =(1-u)^{\alpha-1}\bigl\{(1-u)A'(u)-\alpha A(u)\bigr\}
 \geq0\, ,
\]
and $F$ is non-decreasing on $[0,1)$.  As $t\uparrow1$, $W_t\to X$ in
every $L^r$ with $r<\infty$, and $H$ is a polynomial, so
$F(t)\to F(1)=\E\abs{H(X)}^p$.  Therefore
\begin{equation}
 F(t)\leq F(1)=\E\abs{H(X)}^p\, ,
 \qquad 0\leq t\leq1\, .
 \label{eq:elliptic-monotonicity}
\end{equation}

For $s\neq0$, let $r\deq\cosh 2s$ and $t\deq r^{-1}$.  Then
\[
 \frac{X\cosh s+iY\sinh s}{\sqrt r}
 \stackrel{\mathrm{law}}=W_t\, .
\]
Homogeneity and \eqref{eq:elliptic-monotonicity} prove
\eqref{eq:homogeneous-complex-comparison}; the case $s=0$ is immediate.
\end{proof}

\subsection{Homogenization and complex rotations}
\label{sec:complex-rotations}

We extend the homogeneous comparison to arbitrary polynomials by
homogenization using auxiliary Gaussian variables.  This gives an $L^p$
bound for their complex rotations.

\begin{lemma}[Complex Gaussian rotation]
\label{lem:complex-Gaussian-rotation}
Let $P:\R^d\to\R$ be a polynomial of degree at most $n$, let $p\geq2$, and
let $X,Y$ be independent standard Gaussian vectors in $\R^d$.  For
$z\in\mathbb C$, set
\[
 U_zP(X,Y)\deq P\bigl(X\cos z+Y\sin z\bigr)\, .
\]

This rotation satisfies
\begin{equation*}
 \norm{U_zP}_{L^p(\gamma_d\otimes\gamma_d;\mathbb C)}
 \leq2\,\bigl(\cosh(2\abs{\operatorname{Im}z})\bigr)^{n/2}\,
       \norm P_{L^p(\gamma_d)}\, .
\end{equation*}
The factor $2$ may be omitted if $p$ is an even integer, or if $P$ is even or odd.
\end{lemma}

\begin{proof}
\emph{Reduction to imaginary rotations.}
Write $z\deq a+is$ and define
\[
 X'\deq X\cos a+Y\sin a\, ,
 \qquad
 Y'\deq-X\sin a+Y\cos a\, .
\]
Since $\gamma_d\otimes\gamma_d$ is invariant under orthogonal maps,
$(X',Y')$ has the same law as $(X,Y)$, and
\[
 X\cos z+Y\sin z=X'\cosh s+iY'\sinh s\, .
\]
It therefore suffices to estimate $P(Z_s)$, where
$Z_s=X\cosh s+iY\sinh s$ and $s\in\R$.

\emph{Even and  odd polynomials.}
The assertion is immediate for constant $P$.  Suppose first that $P$ is
non-constant and even or odd.
Write
\[
 P=\sum_{\substack{0\leq j\leq N\\j\equiv N\!\!\!\pmod2}}P_j\, ,
\]
where $P_j$ is homogeneous of degree $j$ and $N\deq\deg P$.  To supply the
missing degree, we multiply $P_j$ by a power of the normalized sum of
squares of auxiliary Gaussian variables; this sum is homogeneous of degree
$2$ and converges to $1$ by the law of large numbers. This trick goes back to Borell~\cite{Bor1978Tail}. More precisely, let
$G=(G_1,\dots,G_L)$ be an auxiliary standard Gaussian vector, independent of
$X,Y$, and put
\[
 S_L(G)\deq\frac1L\sum_{\ell=1}^LG_\ell^2\, ,
 \qquad
 \widetilde P_L(x,G)
 \deq\sum_{\substack{0\leq j\leq N\\j\equiv N\!\!\!\pmod2}}
   P_j(x)\,S_L(G)^{(N-j)/2}\, .
\]
The parity assumption makes every exponent $(N-j)/2$ a non-negative
integer.  Thus, $\widetilde P_L$ is a polynomial homogeneous of total degree
$N$ in $(x,G)$.  If $G'$ is another independent standard Gaussian vector in $\R^L$,
\cref{prop:homogeneous-complex-comparison} gives
\begin{align}
 \norm{\widetilde P_L(
   X\cosh s+iY\sinh s\, ,
   G\cosh s+iG'\sinh s)}_p
 &\leq(\cosh 2s)^{N/2}\,
   \norm{\widetilde P_L(X,G)}_p\, .
 \label{eq:homogenized-comparison}
\end{align}
The laws of large numbers in every finite $L^r$ give
\begin{equation}
 \begin{aligned}
     S_L(G) &\longrightarrow1\, ,\\
 \frac1L\sum_{\ell=1}^L
 (G_\ell\cosh s+iG'_\ell\sinh s)^2
 &\longrightarrow
 \E(G_1\cosh s+iG'_1\sinh s)^2=1\, .
 \end{aligned}
 \label{eq:pseudocovariance-limit}
\end{equation}
Here, $S_L$ uses the holomorphic quadratic form $L^{-1}\sum_{\ell=1}^L z_\ell^2$,
whose mean remains $1$ under the complex rotation by
\eqref{eq:Zs-covariances}.  Put $m_j\deq(N-j)/2$ for the indices in the
definition of $\widetilde P_L$.  Then,
\[
 \widetilde P_L(X,G)-P(X)
 =\sum_j P_j(X)\,\bigl(S_L(G)^{m_j}-1\bigr)\, .
\]
Since $X$ and $G$ are independent, the $L^p$ norm of each summand factors
as $\norm{P_j(X)}_p\,\norm{S_L(G)^{m_j}-1}_p$.
For each fixed integer $m\geq1$, the identity
$a^m-1=(a-1)\sum_{\ell=0}^{m-1}a^\ell$ and H\"older's inequality give
\[
 \norm{S_L(G)^m-1}_p
 \leq\norm{S_L(G)-1}_{2p}
       \sum_{\ell=0}^{m-1}\norm{S_L(G)^\ell}_{2p}
 \longrightarrow0\, .
\]
Indeed, convergence in every finite $L^r$ makes the first factor tend to
zero and keeps the sum bounded; the case $m=0$ contributes zero.
Since only finitely many integral powers occur, the triangle inequality
therefore gives $\widetilde P_L(X,G)\to P(X)$ in $L^p$.
The same argument applies to the polynomial inside the left-hand norm in
\eqref{eq:homogenized-comparison}, replacing $X$ by $Z_s$ and $S_L(G)$ by
\[
 A_L\deq\frac1L\sum_{\ell=1}^L
       (G_\ell\cosh s+iG'_\ell\sinh s)^2\, .
\]
Here $A_L$ is independent of $Z_s$ and tends to $1$ in every finite $L^r$
by \eqref{eq:pseudocovariance-limit}.  The same factorization of powers
and H\"older estimate hold for complex values, so this polynomial
converges to $P(Z_s)$ in $L^p$.
Passing to the limit gives
\begin{equation}
 \norm{P(X\cosh s+iY\sinh s)}_p
 \leq(\cosh 2s)^{N/2}\,\norm{P(X)}_p\, .
 \label{eq:single-parity-comparison}
\end{equation}

\emph{General polynomials.}
Write the projections onto polynomials of even and odd parity,
\[
 P_{\mathrm e}(x)\deq\frac{P(x)+P(-x)}2\, ,
 \qquad
 P_{\mathrm o}(x)\deq\frac{P(x)-P(-x)}2\, .
\]
Since $\gamma_d$ is symmetric, both parity projections are contractions on
$L^p(\gamma_d)$, and both polynomials have degree at most $n$.  Applying
\eqref{eq:single-parity-comparison} separately and using the triangle
inequality gives the factor $2$.

\emph{Even exponents.}
If $p=2q$ with $q\in\mathbb N$, then $P^q$ is a polynomial of degree at
most $nq$.  Applying \cref{lem:complex-chaos} to this polynomial gives
\begin{align*}
 \E\abs{P(Z_s)}^{2q}
 &=\E\abs{P^q(Z_s)}^2\leq(\cosh 2s)^{nq}\,\E\abs{P^q(X)}^2
 =(\cosh 2s)^{nq}\,\E\abs{P(X)}^{2q}\, .
\end{align*}
Taking $2q$-th roots proves the estimate with factor $1$ for even $p$.
\end{proof}

\subsection{Cauchy's formula}
\label{sec:cauchy}

Applying Cauchy's formula to the complex rotations gives the gradient
estimate in \cref{thm:Bernstein-Markov}.  Combining it with \cref{thm:main}
then yields \cref{cor:polynomial-divergence}.

\begin{proof}[Proof of \cref{thm:Bernstein-Markov}]
The map $z\mapsto U_zP$ is a finite sum of entire scalar coefficients
multiplied by polynomials in $(X,Y)$.  It is therefore an entire function with
values in $L^p(\gamma_d\otimes\gamma_d;\mathbb C)$, and
\begin{equation*}
 \frac{\dd}{\dd z}\, U_zP\Big|_{z=0}=Y\cdot DP(X)\, .
\end{equation*}
Conditional on $X$, this is a centered Gaussian with variance
$\abs{DP(X)}^2$.  Thus \cref{lem:Gaussian-linear-forms}(ii) gives
\begin{equation}
 \norm{Y\cdot DP(X)}_p=m_p\,\norm{DP}_p\, .
 \label{eq:conditional-Gaussian-gradient}
\end{equation}
The Banach-valued Cauchy formula of \cref{app:standard-tools} on the
circle $\abs z=\rho$, followed by
Minkowski's inequality and \cref{lem:complex-Gaussian-rotation}, gives
\begin{align*}
 \norm{Y\cdot DP(X)}_p
 &\leq\frac1\rho\sup_{\abs z=\rho}\norm{U_zP}_p \\
 &\leq\frac2\rho\,(\cosh 2\rho)^{n/2}\,\norm P_p
 \leq\frac2\rho\, e^{n\rho^2}\,\norm P_p\, .
\end{align*}
Here we used $\cosh t\leq e^{t^2/2}$.  The factor
$\rho^{-1}e^{n\rho^2}$ is minimized at $\rho=(2n)^{-1/2}$, giving
\begin{equation}
 \norm{Y\cdot DP(X)}_p
 \leq2\sqrt{2en}\,\norm P_p\, .
 \label{eq:directional-polynomial-bound}
\end{equation}
Equations \eqref{eq:directional-polynomial-bound} and
\eqref{eq:conditional-Gaussian-gradient} prove
\eqref{eq:Bernstein-Markov}.  The factor $2$ is absent throughout when the
corresponding refinement in \cref{lem:complex-Gaussian-rotation} applies.

To prove \eqref{eq:Bernstein-Markov-order}, insert the lower bound
$m_p\geq\sqrt{p/e}$ of \cref{lem:Gaussian-moments} into
\eqref{eq:Bernstein-Markov}; this gives the stated constant $2e\sqrt2$.
\end{proof}

\begin{proof}[Proof of \cref{cor:polynomial-divergence}]
Put $P_0\deq P-\E P$, and let $G$ be an auxiliary standard Gaussian vector in
$\R^d$.  Since $p\geq2$, \cref{lem:Gaussian-linear-forms}(iii) gives, pointwise in
$x$,
\[
 \abs{DP(x)}_{\HS}
 =\bigl(\E_G\abs{DP(x)^{\mathsf T}G}^2\bigr)^{1/2}
 \leq\bigl(\E_G\abs{DP(x)^{\mathsf T}G}^p\bigr)^{1/p}\, .
\]
For each fixed $G$, the scalar polynomial
$P_G(x)\deq\ip{G}{P_0(x)}$ has degree at most $n$.  Hence
\begin{align*}
 \norm{DP}_{L^p(\gamma_d;\HS_d)}^p
 &\leq\E_G\norm{DP_G}_{L^p(\gamma_d;\R^d)}^p
 \leq\Bigl(\frac{2\sqrt{2en}}{m_p}\Bigr)^p\,
       \E_G\norm{P_G}_{L^p(\gamma_d)}^p
 =\bigl(2\sqrt{2en}\bigr)^p\,
       \norm{P_0}_{L^p(\gamma_d;\R^d)}^p\, .
\end{align*}
The last identity is \cref{lem:Gaussian-linear-forms}(ii).  Taking $p$-th
roots and applying \cref{thm:main} proves \eqref{eq:polynomial-divergence}, since
$48\,\bigl(2\sqrt{2e}\bigr)=96\sqrt{2e}$.  The last bound follows from
$\abs{\E P}\leq\norm P_{L^p(\gamma_d;\R^d)}$,
$\norm{P-\E P}_{L^p(\gamma_d;\R^d)}\leq2\,\norm P_{L^p(\gamma_d;\R^d)}$,
and $\sqrt p\leq p\sqrt n$.
In the intermediate estimate the factor $2$ may be omitted if $p$ is even or if the components of
$P_0$ have one common parity.
\end{proof}

\section{Tightness of the two inequalities}
\label{sec:tightness}

\subsection{The divergence inequality}

\begin{proof}[Proof of the converse assertion in \cref{thm:main}]
Suppose that \eqref{eq:main-converse} holds, and let $e_1$ be the first
unit vector of $\R^d$.  For the mean term, take $V\equiv e_1$.  Then
$DV=0$, $\abs{\E V}=1$, and $\delta V=X_1$, so \eqref{eq:main-converse}
and \cref{lem:Gaussian-moments} give
\[
 A\geq\norm{\delta V}_{L^p}=m_p\geq\sqrt{\frac pe}\, .
\]

For the derivative term, take $V(x)=x_1e_1$.  Then $\E V=0$,
$\abs{DV}_{\HS}\equiv1$, and $\delta V=X_1^2-1$, so
$B\geq\norm{X_1^2-1}_p$.  Writing $X=X_1$, the reverse triangle inequality
and \cref{lem:Gaussian-moments} give
\[
 \norm{X^2-1}_p
 \geq\norm X_{2p}^2-1
 \geq\frac{2p}{e}-1\, .
\]
For $p\geq e$, this is at least $p/e$.  For $2\leq p\leq e$,
$\norm{X^2-1}_p\geq\norm{X^2-1}_2=\sqrt2\geq p/e$.
Thus $B\geq p/e$, and one may take $c=1/e$ in both lower bounds.
\end{proof}

\subsection{The Bernstein--Markov inequality}
\label{sec:BM-tightness}

\begin{proposition}[Sharpness]
\label{prop:Bernstein-Markov-sharpness}
For polynomials of degree at most $n\geq1$, the optimal coefficient in the
gradient estimate at $p=2$ is $\sqrt n$.
For every $N\geq1$ and $p\geq2$, the polynomial $P_N(x)\deq\prod_{j=1}^Nx_j$
satisfies the non-asymptotic lower bound
\begin{equation}
 \frac{\norm{DP_N}_{L^p(\gamma_N;\R^N)}}
      {\norm{P_N}_{L^p(\gamma_N)}}
 \geq\sqrt{\frac{N}{p-1}}\, .
 \label{eq:product-sharpness-finite}
\end{equation}
Moreover, for every fixed $p>1$,
\begin{equation}
 \lim_{N\to\infty}
 \frac{\norm{DP_N}_{L^p(\gamma_N;\R^N)}}
      {\sqrt N\,\norm{P_N}_{L^p(\gamma_N)}}
 =\frac1{\sqrt{p-1}}\, .
 \label{eq:tensor-product-sharpness}
\end{equation}
Thus, \eqref{eq:Bernstein-Markov-order} has the optimal joint order for
$p\geq2$.
\end{proposition}

\begin{proof}
If $P$ belongs to the $n$-th Wiener chaos, Gaussian integration
by parts gives
\[
 \norm{DP}_2^2=n\,\norm P_2^2\, ,
\]
which proves the first assertion.

For the remaining assertions, let $G_1,\dots,G_N$ be independent standard
Gaussians.  Since
the zero set of $P_N$ is null,
\[
 \norm{DP_N(G)}
 =\abs{P_N(G)}\,\Bigl(\sum_{j=1}^NG_j^{-2}\Bigr)^{1/2}
 \quad\text{almost surely}\, .
\]
Let $\nu_p$ be the probability measure on $\R$ defined by
\[
 \nu_p(\dd x)\deq\frac{\abs x^p}{m_p^p}\,\gamma_1(\dd x)\, ,
\]
and set $Z(x)\deq x^{-2}$.  A change of measure gives
\begin{equation}
 \frac{\norm{DP_N}_p}{\sqrt N\,\norm{P_N}_p}
 =\Bigl[
   \E_{\nu_p^{\otimes N}}
   \Bigl(\frac1N\sum_{j=1}^NZ_j\Bigr)^{p/2}
  \Bigr]^{1/p}\, .
 \label{eq:tilted-product-identity}
\end{equation}
The moment recurrence \eqref{eq:moment-recurrence} gives
\[
 \E_{\nu_p}Z
 =\frac{\E\abs g^{p-2}}{\E\abs g^p}
 =\frac1{p-1}\, ,
 \qquad
 \E_{\nu_p}Z^{p/2}=\frac1{m_p^p}<\infty\, .
\]
For $p\geq2$, Jensen's inequality gives
\[
 \E_{\nu_p^{\otimes N}}
 \Bigl(\frac1N\sum_{j=1}^NZ_j\Bigr)^{p/2}
 \geq\bigl(\E_{\nu_p}Z\bigr)^{p/2}
 =(p-1)^{-p/2}\, .
\]
Together with \eqref{eq:tilted-product-identity}, this proves
\eqref{eq:product-sharpness-finite} for every degree $N$ and exponent
$p\geq2$.

If $p\geq2$, the $L^{p/2}$ law of large numbers applies.  If $1<p<2$, the
$L^1$ law of large numbers and Jensen's inequality give, with
$A_N\deq N^{-1}\sum_{j=1}^NZ_j$ and $a\deq(p-1)^{-1}$,
\[
 \E\abs{A_N-a}^{p/2}
 \leq\bigl(\E\abs{A_N-a}\bigr)^{p/2}\longrightarrow0\, .
\]
Thus, $A_N$ converges in $L^{p/2}$ to $a$.  Passing to the limit in
\eqref{eq:tilted-product-identity} is immediate when $p\geq2$; when $p<2$ it
follows from
$\abs{x^{p/2}-y^{p/2}}\leq\abs{x-y}^{p/2}$ for $x,y\geq0$.  This proves
\eqref{eq:tensor-product-sharpness}.
\end{proof}

At $p=1$, the same product polynomials disprove
\cref{conj:Eskenazis-Ivanisvili}.  Indeed,
\eqref{eq:tilted-product-identity} remains valid with
$\nu_1(\dd x)=\abs x\,\gamma_1(\dd x)/m_1$, whereas
$\E_{\nu_1}Z=m_1^{-1}\int_{\R}\abs x^{-1}\,\dd\gamma_1=\infty$.
For each $L>0$, apply the law of large numbers and bounded convergence
to the square root of the average of $\min\{Z_j,L\}$.  Since
$Z_j\geq\min\{Z_j,L\}$, this gives
\[
 \liminf_{N\to\infty}
 \frac{\norm{DP_N}_{L^1(\gamma_N;\R^N)}}
      {\sqrt N\,\norm{P_N}_{L^1(\gamma_N)}}
 \geq\bigl(\E_{\nu_1}\min\{Z,L\}\bigr)^{1/2}\, .
\]
Letting $L\to\infty$ shows that the ratio diverges, so no dimension-free
constant exists at $p=1$.

The lower bound in \eqref{eq:tensor-product-sharpness} is the reason for the
denominator $p-1$ in the conjectural estimate for $1<p<2$.  The proof of
\cref{thm:Bernstein-Markov} cannot simply be continued into that range.  To
see this directly, take $M=1$, $H(z)=z$, and $q=p/2<1$ in
\cref{lem:Fock-recursion}.  Then,
\[
 N_0=\Gamma(q+1)\, ,
 \qquad
 N_1=\Gamma(q+3)\, ,
 \qquad
 \alpha=q+\frac12\, .
\]
The asserted recursion would require
\[
 (q+1)(q+2)\geq4q+2\, ,
\]
which is equivalent to $q(q-1)\geq0$ and is false.  Correspondingly, for
small $s$, if $X,Y$ are independent standard real Gaussians,
\[
 \log\frac{\norm{X\cosh s+iY\sinh s}_p}{\norm X_p}
 =\frac{p}{2(p-1)}\,s^2+o(s^2)\, ,
\]
whereas
$\log(\cosh(2s)^{1/2})=s^2+o(s^2)$.  The homogeneous comparison used in
\cref{prop:homogeneous-complex-comparison} therefore fails immediately below
$p=2$.  To justify the first expansion, differentiation at $s^2=0$ is
dominated when $p>1$, and the recurrence \eqref{eq:moment-recurrence}
gives
\[
 \E\abs{X\cosh s+iY\sinh s}^p
 =\E\abs X^p\,
  \Bigl\{1+\frac{p^2}{2(p-1)}\,s^2+o(s^2)\Bigr\}\, .
\]
Proving the expected $\sqrt{n/(p-1)}$ estimate will require a different
comparison principle.

\newpage
\appendix

\section{Analytic tools and supplementary proofs}
\label{app:analytic-tools}

This appendix collects the analytic material used in the proofs.  We first
record standard analytic tools and Gaussian moment computations, then derive
a Hilbert-valued Gaussian randomization estimate.  We finish with the
Bochner identity, compactness, and the zero-set lemma used in the
Calder\'on--Zygmund decomposition.

\subsection{Standard analytic tools}
\label{app:standard-tools}

\paragraph{Banach-valued Cauchy formula.}
Let $B$ be a complex Banach space and let $\Phi:\mathbb C\to B$ be entire,
that is, $\Phi(z)=\sum_{k\geq0}b_kz^k$ with $b_k\in B$ and
$\sum_{k\ge 0} \norm{b_k}_B\,\rho^k<\infty$ for every $\rho>0$.  Then, for every
$\rho>0$,
\[
 \Phi'(0)=\frac1{2\pi i}\oint_{\abs z=\rho}\frac{\Phi(z)}{z^2}\,\dd z\, ,
 \qquad\text{so that}\qquad
 \norm{\Phi'(0)}_B\leq\frac1\rho\sup_{\abs z=\rho}\norm{\Phi(z)}_B\, .
\]
In \cref{sec:proof-Bernstein-Markov} this is applied with
$B\deq L^p(\gamma_d\otimes\gamma_d;\mathbb C)$ and
$\Phi(z)\deq P(X\cos z+Y\sin z)$ for a polynomial $P$; expanding
$P(X\cos z+Y\sin z)$ in monomials shows that $\Phi$ is a finite sum of entire
scalar functions of $z$ multiplied by fixed elements of $B$, hence entire in
the above sense.

\paragraph{Gaussian concentration.}
If $f:\R^d\to\R$ is $L$-Lipschitz, then $f\in L^1(\gamma_d)$ and
$\gamma_d\{f-\E f\geq t\}\leq e^{-t^2/(2L^2)}$ for every $t\geq0$
\cite{Ledoux2001}.  This is used only in \cref{app:bounded-derivative}.

\subsection{Gaussian moments and linear forms}
\label{app:moments}

We use the formulas and bounds below to track the constants in
both main inequalities and in the tightness examples.

\begin{lemma}[Gaussian moments]\label{lem:Gaussian-moments}
For every $p>0$,
\begin{equation*}
 m_p^p\deq \E\abs g^p=\frac{2^{p/2}}{\sqrt\pi}\,\Gamma\Bigl(\frac{p+1}{2}\Bigr)\,,
\end{equation*}
and for $p>1$,
\begin{equation}\label{eq:moment-recurrence}
 \E\abs g^p=(p-1)\,\E\abs g^{p-2}\, .
\end{equation}
For every $p\geq2$,
\begin{equation}\label{eq:mp-two-sided}
 \sqrt{\frac pe}\leq m_p\leq\sqrt{p-1}\, ;
\end{equation}
in particular $c\sqrt p\leq m_p\leq C\sqrt p$ on $[2,\infty)$ with universal
constants $c,C>0$.
\end{lemma}

\begin{proof}
The Gaussian moment formula is standard; see
\cite[Exercise~2.5.1]{Ver18HighDimProb}.

For the upper bound in \eqref{eq:mp-two-sided}, the case $p=2$ is $m_2=1$.
For $p>2$, the recurrence and monotonicity of the Gaussian $L^p$ norms give
$m_p^p=(p-1)\,m_{p-2}^{p-2}\leq(p-1)\,m_p^{p-2}$, hence $m_p^2\leq p-1$.
For the lower bound, Stirling's inequality
\cite[Eq.~(5.6.1)]{NISTDLMF}, applied at $(p+1)/2$, gives
\[
 m_p\geq\Bigl(\frac2e\Bigr)^{1/(2p)}\,
          \sqrt{\frac{p+1}{e}}
     \geq\sqrt{\frac pe}\, .
\]
The last inequality follows from Bernoulli's inequality,
$(1+1/p)^p\geq2>e/2$ for $p\geq1$.
\end{proof}

By H\"older's inequality, $p\mapsto m_p$ is non-decreasing on $(0,\infty)$,
and $m_2=1$.

The following identities relate Gaussian linear forms to Euclidean norms.
In particular, the conditional identity converts the estimate for
$Y\cdot DP(X)$ into the gradient bound \eqref{eq:Bernstein-Markov}.

\begin{lemma}[Gaussian linear forms]\label{lem:Gaussian-linear-forms}
Let $G\sim\gamma_d$ and $0<p<\infty$.
\begin{enumerate}[label=\textup{(\roman*)}]
\item For $v\in\R^d$, the variable $\ip Gv$ is a centered Gaussian with
variance $\abs v^2$.  Consequently,
\[
 \norm{\ip Gv}_{L^p(\gamma_d)}=m_p\,\abs v\, .
\]
\item If $v:\R^d\to\R^d$ is measurable and $X\sim\gamma_d$ is independent
of $G$, then
\[
 \norm{\ip{G}{v(X)}}_{L^p(\gamma_d\otimes\gamma_d)}
 =m_p\,\norm v_{L^p(\gamma_d;\R^d)}\, .
\]
\item For a $d\times d$ matrix $A$, $\E\abs{A^{\mathsf T}G}^2=\abs A_{\HS}^2$.
Hence, for $p\geq2$,
$\abs A_{\HS}\leq(\E\abs{A^{\mathsf T}G}^p)^{1/p}$.
\end{enumerate}
\end{lemma}

\begin{proof}
(i) is the stability of Gaussian laws under linear maps.  (ii) follows by
conditioning on $X$ and applying (i) with $v=v(X)$.  For (iii),
$\E\abs{A^{\mathsf T}G}^2=\E\,G^{\mathsf T}AA^{\mathsf T}G
=\operatorname{tr}(AA^{\mathsf T})=\abs A_{\HS}^2$, and the inequality follows from
Jensen's inequality.
\end{proof}

\subsection{Gaussian Hilbert randomization}
\label{app:randomization}

The estimate below passes scalar $L^q$ bounds to Hilbert-valued ones
without any loss in $\dim \cH$; it is used in \cref{lem:Hilbert-extension}.

\begin{lemma}[Gaussian Hilbert randomization]\label{lem:gaussian-randomization}
Let $\cH$ be a finite-dimensional Hilbert space, let $y_1,\dots,y_m\in \cH$, and let
$g_1,\dots,g_m$ be independent standard Gaussians.  For every $1\le q\le2$,
\begin{equation*}
 m_q\,\biggl(\sum_{a=1}^m\norm{y_a}_\cH^2\biggr)^{1/2}
 \le
 \biggl(\E_g\norm{\sum_{a=1}^m g_a y_a}_\cH^q\biggr)^{1/q}
 \le
 \biggl(\sum_{a=1}^m\norm{y_a}_\cH^2\biggr)^{1/2}\, ,
\end{equation*}
where $m_q\geq m_1=\sqrt{2/\pi}$.  Moreover, for every $r\ge2$,
\begin{equation*}
 \biggl(\E_g\norm{\sum_{a=1}^m g_a y_a}_\cH^r\biggr)^{1/r}
 \leq m_r\,\biggl(\sum_{a=1}^m\norm{y_a}_\cH^2\biggr)^{1/2}
 \leq\sqrt r\,\biggl(\sum_{a=1}^m\norm{y_a}_\cH^2\biggr)^{1/2}\, .
\end{equation*}
\end{lemma}

\begin{proof}
    Put $Y\deq\sum_{a=1}^m g_a y_a$ and
    $\sigma^2\deq\E_g\norm Y_\cH^2=\sum_{a=1}^m \norm{y_a}_\cH^2$.
The Gaussian moment comparison of Lata\l a and Oleszkiewicz
\cite[Corollary~3]{LatalaOleszkiewicz1999} states that
\[
 \norm Y_{L^u(g;\cH)}
 \leq\frac{m_u}{m_v}\,\norm Y_{L^v(g;\cH)}\, ,
 \qquad u\geq v>0\, .
\]
Taking $(u,v)=(2,q)$ and using $m_2=1$ gives
$\norm Y_{L^q(g;\cH)}\geq m_q\,\sigma$; the upper bound
$\norm Y_{L^q(g;\cH)}\leq\sigma$ follows from $q\leq2$.
Taking $(u,v)=(r,2)$ gives
$\norm Y_{L^r(g;\cH)}\leq m_r\,\sigma$, and
$m_r\leq\sqrt{r-1}\leq\sqrt r$ by \cref{lem:Gaussian-moments}.
\end{proof}

\subsection{Tools for the Calder\'on--Zygmund decomposition}
\label{app:bochner-compactness}

The three lemmas below are used in the Calder\'on--Zygmund decomposition
of \cref{sec:CZ}.  Compactness enters only to extract a limit of the
penalized minimizers at fixed $d$; the bounds in the resulting
decomposition are uniform in $d$.

\begin{lemma}[Bochner identity]\label{lem:bochner}
For every $u\in\Dom(\cN)$,
\begin{equation}\label{eq:bochner}
 \norm{\cN u}_{L^2(\gamma_d)}^2
 =\norm{D^2u}_{L^2(\gamma_d;\HS_d)}^2
  +\norm{Du}_{L^2(\gamma_d;\R^d)}^2\, .
\end{equation}
In particular, $u$ has a weak Hessian in $L^2(\gamma_d;\HS_d)$, and
$u\in W^{2,2}(\gamma_d)$.
\end{lemma}

This is the $\Gamma_2$ identity of the Ornstein--Uhlenbeck semigroup, whose
curvature is $1$ \cite[Section~2.7]{BGL14}.  We include
the short verification for completeness.

\begin{proof}
For a polynomial $u$, \cref{lem:IBP} and
$D_i\cN=(\cN+1)D_i$ give
\begin{align*}
 \norm{\cN u}_2^2
 &=\E\ip{D\cN u}{Du}
 =\sum_{i=1}^d\E[D_iu\,(\cN+1)D_iu]\\
 &=\sum_{i=1}^d\bigl(\norm{DD_iu}_2^2+\norm{D_iu}_2^2\bigr)
 =\norm{D^2u}_{L^2(\HS_d)}^2+\norm{Du}_2^2\, ,
\end{align*}
where the third equality uses \cref{lem:IBP} once more in the form
$\E[v\,\cN v]=\norm{Dv}_2^2$.  For $u\in\Dom(\cN)$, apply this identity to the
finite chaos sums $u^{(N)}=\sum_{n\leq N}u_n$.  Since
$\cN u^{(N)}\to\cN u$ in $L^2$, the identity applied to differences shows
that $Du^{(N)}$ and $D^2u^{(N)}$ are Cauchy in their respective $L^2$
spaces.  Their limits are the weak first and second derivatives of $u$.
Passing to the limit proves \eqref{eq:bochner}.
\end{proof}

\begin{lemma}[Compactness used in the Calder\'on--Zygmund decomposition]\label{lem:compact}
For fixed $d$, every sequence bounded in $W^{1,2}(\gamma_d)$ has an
$L^2(\gamma_d)$-convergent subsequence.
\end{lemma}

\begin{proof}
By \eqref{eq:Du-chaos}, for every integer $N\geq1$ and every
$u\in W^{1,2}(\gamma_d)$, $\sum_{n>N}\norm{u_n}_2^2\leq N^{-1}\,\norm{Du}_2^2$.
For fixed $d$ and $N$, the sum $\bigoplus_{n\leq N}\mathcal C_n$ is
finite-dimensional.  A diagonal subsequence argument, followed by the
uniform tail estimate above, proves the claim.
\end{proof}

\begin{lemma}[First and second derivatives vanish on a zero set]\label{lem:zero-set}
Let $u\in W^{2,2}_{\mathrm{loc}}(\R^d)$.  Then
\[
 Du=0\quad\text{a.e. on }\{u=0\}\, ,
 \qquad
 D^2u=0\quad\text{a.e. on }\{u=0\}\, .
\]
The same conclusion holds for $u\in\Dom(\cN)$, with respect to $\gamma_d$.
\end{lemma}

\begin{proof}
By \cite[Lemma~7.7]{GilbargTrudinger2001}, the weak gradient of a
Sobolev function vanishes almost everywhere on each of its level sets.
Applying this result on bounded balls to $u$ gives $Du=0$ almost everywhere
on $\{u=0\}$.  For each $j=1,\dots,d$, we have
$\partial_j u\in W^{1,2}_{\mathrm{loc}}(\R^d)$, so a second application gives
$D(\partial_j u)=0$ almost everywhere on $\{\partial_j u=0\}$.
Since $\{u=0\}\subset\{\partial_j u=0\}$ up to a null set, this proves
$D^2u=0$ almost everywhere on $\{u=0\}$.

If $u\in\Dom(\cN)$, \cref{lem:bochner} gives
$u\in W^{2,2}(\gamma_d)\subset W^{2,2}_{\mathrm{loc}}(\R^d)$, since the
density of $\gamma_d$ is bounded below by a positive constant on each
bounded ball.  The preceding argument therefore applies, and Lebesgue-null
and $\gamma_d$-null sets coincide.
\end{proof}

\section{A direct proof of the bounded derivative inequality}
\label{app:bounded-derivative}

We give a direct change-of-variables proof of the bounded derivative
consequence of \cref{thm:main}.

\begin{corollary}[Bounded derivative inequality]
\label{cor:bounded-derivative}
Let $V:\R^d\to\R^d$ belong to
$W^{1,p_0}(\gamma_d;\R^d)$ for some $p_0\geq2$, and suppose that
\[
 L\deq\norm{DV}_{L^\infty(\gamma_d;\HS_d)}<\infty\, .
\]
Then, for every $p\geq2$,
\begin{equation}
 \norm{\delta V}_{L^p(\gamma_d)}
 \leq C\sqrt p\,\abs{\E V}+CpL\, .
 \label{eq:bounded-derivative}
\end{equation}
If $\E V=0$, then
\begin{equation}
 \Pp\bigl\{\abs{\delta V}\geq t\bigr\}
 \leq2\exp\Bigl\{-\,\frac{ct}{L}\Bigr\}\, ,
 \qquad t\geq0\, .
 \label{eq:bounded-derivative-tail}
\end{equation}
The assertions have their usual interpretation when $L=0$, and the constants
are universal.
\end{corollary}

\begin{proof}
We first suppose that $V$ is $C^1$ with bounded derivative.  Put
$m\deq\E V$ and $W\deq V-m$.  The Gaussian Poincar\'e inequality
\eqref{eq:Poincare}, applied componentwise, gives
\begin{equation}
 \E\abs W^2
 =\sum_{i=1}^d\operatorname{Var}(V_i)
 \leq\E\abs{DV}_{\HS}^2
 \leq L^2\, .
 \label{eq:vector-Poincare-bdd}
\end{equation}
Moreover, $x\mapsto\abs{W(x)}$ is $L$-Lipschitz.  Gaussian concentration
(\cref{app:standard-tools}), \eqref{eq:vector-Poincare-bdd}, and integration of the resulting tail show
that
\begin{equation}
 \sup_{\abs\lambda L\leq1/4}
 \E\exp\Bigl\{\frac{\lambda^2}{2}\,\abs W^2\Bigr\}
 \leq C_0\, .
 \label{eq:W-square-exponential}
\end{equation}
Indeed, if $Z\deq(\abs W-\E\abs W)_+$, then
$\Pp\{Z\geq t\}\leq e^{-t^2/(2L^2)}$, while
$\E\abs W\leq L$ and $\abs W^2\leq2L^2+2Z^2$.

Fix $\lambda$ with $\abs\lambda L<1$.  The map
\[
 T_\lambda : x \longmapsto x-\lambda W(x)
\]
is a global $C^1$ diffeomorphism: for each $y$, the equation
$x=y+\lambda W(x)$ has a unique solution by the contraction principle.
Moreover, $DT_\lambda=I-\lambda DW$ is invertible, so the inverse function
theorem gives a $C^1$ inverse.  Its determinant is positive: for each $x$,
$I-t\lambda DW(x)$ remains invertible for $0\leq t\leq1$, and its
determinant starts at $1$ and cannot change sign.
The change of variables $y=T_\lambda(x)$ gives
\begin{align*}
 1
 &=(2\pi)^{-d/2}\int_{\R^d}
   e^{-\abs{x-\lambda W(x)}^2/2}
   \det(I-\lambda DW(x))\,\dd x
 =\E\bigl[e^{\lambda X\cdot W-\lambda^2\abs W^2/2}
   \det(I-\lambda DW)\bigr]\, .
\end{align*}
Using $X\cdot W=\delta W+\operatorname{tr}(DW)$ yields
\begin{equation*}
 1=\E\Bigl[\exp\Bigl\{\lambda\delta W
          -\frac{\lambda^2}{2}\,\abs W^2\Bigr\}\,J_\lambda\Bigr]\, ,
 \qquad
 J_\lambda
 \deq e^{\lambda\operatorname{tr}(DW)}
   \det(I-\lambda DW)\, .
\end{equation*}
If $B\deq\lambda DW$ and $\norm B_{\mathrm{op}}\leq1/2$, then
\[
 \log J_\lambda
 =-\sum_{k=2}^\infty\frac{\operatorname{tr}(B^k)}{k}\, .
\]
Furthermore,
\[
 \abs{\operatorname{tr}(B^k)}
 \leq\abs B_{\HS}^2\,\norm B_{\mathrm{op}}^{k-2}\, .
\]
It follows that $J_\lambda\geq e^{-\lambda^2L^2}$ whenever
$\abs\lambda L\leq1/2$, and therefore
\begin{equation}
 \E\exp\Bigl\{\lambda\delta W
          -\frac{\lambda^2}{2}\,\abs W^2\Bigr\}
 \leq e^{\lambda^2L^2}\, .
 \label{eq:damped-divergence-mgf}
\end{equation}
Cauchy--Schwarz, \eqref{eq:W-square-exponential}, and
\eqref{eq:damped-divergence-mgf} yield
\[
 \E e^{\lambda\delta W/2}
 \leq
 \bigl(\E e^{\lambda\delta W-\lambda^2\abs W^2/2}\bigr)^{1/2}\,
 \bigl(\E e^{\lambda^2\abs W^2/2}\bigr)^{1/2}
 \leq C_1
\]
for both signs of $\lambda$ whenever $\abs\lambda L\leq1/4$.  Taking
$\lambda=\pm(4L)^{-1}$ and using Chernoff's bound proves
\eqref{eq:bounded-derivative-tail}, after adjusting the constant
for small $t$.  Integrating the tail gives
$\norm{\delta W}_{L^p}\leq CpL$.  Finally,
\[
 \delta V=\delta W+\ip{X}{m}\, ,
 \qquad
 \norm{\ip{X}{m}}_{L^p}=m_p\,\abs m\lesssim\sqrt p\,\abs m\, ,
\]
which proves \eqref{eq:bounded-derivative}.

For the stated Sobolev formulation, use the Lipschitz representative supplied
by the bounded weak derivative and smooth it with the Ornstein--Uhlenbeck
semigroup.  This representative has at most linear growth and hence belongs
to every finite Gaussian $L^r$ space.  By \cref{lem:OU}, the mean is preserved,
$D(e^{-t\cN}W)=e^{-t}e^{-t\cN}DW$, so the derivative bound does not
increase.  Strong continuity of the semigroup, applied to $W$ and $DW$,
gives $e^{-t\cN}W\to W$ in $W^{1,2}$ as $t\downarrow0$.
By \eqref{eq:delta-L2}, the divergences converge in $L^2$ and, along a
subsequence, almost everywhere.  Fatou's lemma preserves each finite
moment estimate.  Optimizing Markov's
inequality over $r\geq2$ then gives the same tail bound.  This completes the
approximation argument.
\end{proof}

\section{Further comparisons with related work}
\label{app:related-work}

This appendix develops the comparisons mentioned in the introduction:
Gaussian Riesz transforms, divergence inequalities, and their connections
with variational methods and concentration.

\paragraph{First-order comparisons.}
Meyer's work \cite{Meyer1984} develops the first-order comparison recalled
in the introduction and its higher-order analogues.  Bakry
\cite{Bakry1985a,Bakry1985b} placed the argument in the broader setting of
symmetric diffusion semigroups, and Gundy~\cite{Gundy1986} gave a further
probabilistic proof.

Pisier~\cite{Pisier1988} found a notably direct analytic proof.  He embeds the
Gaussian problem into the rotation
$f(x\cos t+y\sin t)$, projects in the $y$ variable onto the first Wiener
chaos, and invokes Calder\'on transference for a one-dimensional singular
integral.  This makes the relation with the Hilbert transform explicit and
gives quantitative information on the constants.  For the first-order
comparison his proof gives an upper bound of order $q$ as $q\to\infty$ and
of order $(q-1)^{-3/2}$ as $q\downarrow1$.

For the first-order transform, subsequent work gave bounds of order
$q$ as $q\to\infty$ and $(q-1)^{-1}$ as $q\downarrow1$: see Arcozzi
\cite{Arcozzi1998} and the later martingale and Bellman
function estimates in
\cite{DragicevicVolberg2006,BanuelosOsekowski2015}.  Larsson-Cohn
\cite{LarssonCohn2002} proved matching lower bounds in both limits for
the two Meyer constants $C_q$ and $c_q^{-1}$.

\paragraph{Higher-order transforms.}
Pisier~\cite{Pisier1988} also carried the argument through for
the full iterated gradients of every fixed odd order, obtaining an upper
bound of order $q$ as $q\to\infty$ and an upper bound of order
$(q-1)^{-1-k/2}$ as $q\downarrow1$ for order $k$.  The scalar coefficient in
his calculation vanishes when $k$ is even.  Guti\'errez, Segovia, and Torrea
\cite{GutierrezSegoviaTorrea1996} later established dimension-free
full vector bounds by higher-order Littlewood--Paley theory, while Forzani,
Scotto, and Urbina \cite{ForzaniScottoUrbina2001} gave a short multiplier
proof for each fixed multi-index.  These qualitative results do not track the
optimal dependence as $q\downarrow1$ for the full tensor norm considered here.

For orders at least three, Forzani and Scotto
\cite{ForzaniScotto1998} proved that the standard Gaussian Riesz transforms
fail to be of weak type $(1,1)$ already on the line; Casarino, Ciatti, and
Sj\"ogren \cite{CasarinoCiattiSjogren2021} established the corresponding
dichotomy for general Ornstein--Uhlenbeck semigroups.  For $r\geq2$, the
upper bound in \cref{cor:even-Riesz} follows by composing $r$ second-order
estimates.

\paragraph{Weak-type estimates and variational decompositions.}
In fixed dimension, weak-type $(1,1)$ estimates for the first-order
Gaussian Riesz transforms were proved by Muckenhoupt
\cite{Muckenhoupt1969} on the line and by Fabes, Guti\'errez, and Scotto
\cite{FabesGutierrezScotto1994} in arbitrary dimension.

Coordinatewise weak-type $(1,1)$ estimates for second-order Gaussian
Riesz transforms
\cite{GarciaCuervaMauceriSjogrenTorrea1999,CasarinoCiattiSjogren2021}
do not immediately give a bound for the full Hilbert--Schmidt norm.  The
space $L^{1,\infty}$ is not
normable, so the coordinate estimates cannot simply be squared and summed
while retaining a dimension-free constant.  Dimension-free weak-type
bounds also depend on the operator: Aldaz \cite{Aldaz2011} proved that the
weak-type constants of the cubic maximal operator diverge with the
dimension, whereas Sj\"ogren's weak-type theorem for the
Ornstein--Uhlenbeck maximal operator \cite{Sjogren1983} allows its constant
to depend on the dimension.

In Euclidean space, Stein~\cite{Stein1983} proved dimension-free strong
estimates for Riesz transforms and later posed the weak-type problem at
the 1986 International Congress \cite{Stein1986}.  The dimension-free
weak-type $(1,1)$ theorem of Ouyang, Spector, and Stockdale
\cite{OuyangSpectorStockdale2026}, recalled in the introduction, follows
earlier quantitative work on the dimensional dependence
\cite{Janakiraman2004,SpectorStockdale2021}.  Their use of an obstacle
problem to construct a variational Calder\'on--Zygmund decomposition is
related to partial balayage, divisible sandpiles, and Lewy--Stampacchia
estimates \cite{GustafssonSakai1994,LevinePeres2009,ServadeiValdinoci2013}.
The same fractional obstacle approach has since yielded weak-type $(1,1)$ bounds for the full first-order Riesz transform with constant drift, uniformly in the dimension and drift~\cite{Mukherjee2026Drift}, and for the full horizontal Riesz transform on stratified Lie groups, uniformly in the group structure~\cite{MaoWangZhang2026Stratified}.

The Gaussian Calder\'on--Zygmund theory of Mauceri and Meda
\cite{MauceriMeda2007} uses
admissible balls and dimension-dependent constants.  Our variational
construction gives estimates independent of $d$ by working directly with
the Gaussian Dirichlet form.

\paragraph{Bounded Hessians and concentration.}
For gradient fields $V=DG$, the bounded derivative consequence of
\cref{thm:main} gives moment and tail estimates for $\cN G$ in terms of
$\abs{\E DG}$ and a uniform Hilbert--Schmidt bound on $D^2G$.
Functions with bounded Hessian
also appear in the second-order concentration estimates of Adamczak and
Wolff~\cite{AdamczakWolff2015}.  An independent change-of-variables proof
of the bounded derivative inequality for arbitrary vector fields is given
in \cref{app:bounded-derivative}.

\bibliographystyle{plain}
{\small
\bibliography{references}

@article{AdamczakWolff2015,
  author = {Adamczak, Rados\l{}aw and Wolff, Pawe\l{}},
  title = {Concentration inequalities for non-{L}ipschitz functions with bounded derivatives of higher order},
  journal = {Probability Theory and Related Fields},
  volume = {162},
  pages = {531--586},
  year = {2015},
}

@article{Aldaz2011,
  author = {Aldaz, Jes\'us M.},
  title = {The weak type $(1,1)$ bounds for the maximal function associated to cubes grow to infinity with the dimension},
  journal = {Annals of Mathematics},
  volume = {173},
  pages = {1013--1023},
  year = {2011},
}

@article{Arcozzi1998,
  author = {Arcozzi, Nicola},
  title = {{Riesz} transforms on compact {Lie} groups, spheres and {Gauss} space},
  journal = {Arkiv f\"or Matematik},
  volume = {36},
  pages = {201--231},
  year = {1998},
}

@incollection{Bakry1985a,
  author = {Bakry, Dominique},
  title = {Transformation de {Riesz} pour les semi-groupes sym\'etriques. Premi\`ere partie: \`etude de la dimension 1},
  booktitle = {S\'eminaire de Probabilit\'es XIX},
  series = {Lecture Notes in Mathematics},
  volume = {1123},
  pages = {130--144},
  publisher = {Springer},
  year = {1985},
}

@incollection{Bakry1985b,
  author = {Bakry, Dominique},
  title = {Transformation de {Riesz} pour les semi-groupes sym\'etriques. Seconde partie: \`etude sous la condition {$\Gamma_2\geq0$}},
  booktitle = {S\'eminaire de Probabilit\'es XIX},
  series = {Lecture Notes in Mathematics},
  volume = {1123},
  pages = {145--174},
  publisher = {Springer},
  year = {1985},
}

@book {BGL14,
    AUTHOR = {Bakry, Dominique and Gentil, Ivan and Ledoux, Michel},
     TITLE = {Analysis and geometry of {M}arkov diffusion operators},
    SERIES = {Grundlehren der Mathematischen Wissenschaften [Fundamental
              Principles of Mathematical Sciences]},
    VOLUME = {348},
 PUBLISHER = {Springer, Cham},
      YEAR = {2014},
     PAGES = {xx+552},
}

@article{BanuelosOsekowski2015,
  author = {Ba\~nuelos, Rodrigo and Os\c{e}kowski, Adam},
  title = {Sharp martingale inequalities and applications to {Riesz} transforms on manifolds, {Lie} groups and {Gauss} space},
  journal = {Journal of Functional Analysis},
  volume = {269},
  pages = {1652--1713},
  year = {2015},
}

@article{Bargmann1961,
  author = {Bargmann, Valentine},
  title = {On a {Hilbert} space of analytic functions and an associated integral transform. Part {I}},
  journal = {Communications on Pure and Applied Mathematics},
  volume = {14},
  pages = {187--214},
  year = {1961},
}

@article{Bernstein1912,
  author = {Bernstein, Sergei N.},
  title = {Sur l'ordre de la meilleure approximation des fonctions continues par des polyn\^omes de degr\'e donn\'e},
  journal = {M\'emoires de l'Acad\'emie Royale de Belgique},
  volume = {4},
  pages = {1--103},
  year = {1912},
}

@book{Bogachev1998,
  author = {Bogachev, Vladimir I.},
  title = {{Gaussian} measures},
  series = {Mathematical Surveys and Monographs},
  volume = {62},
  publisher = {American Mathematical Society},
  year = {1998},
}

@book{BorweinErdelyi1995,
  author = {Borwein, Peter and Erd\'elyi, Tam\'as},
  title = {Polynomials and polynomial inequalities},
  series = {Graduate Texts in Mathematics},
  volume = {161},
  publisher = {Springer},
  year = {1995},
}

@article{CasarinoCiattiSjogren2021,
  author = {Casarino, Valentina and Ciatti, Paolo and Sj\"ogren, Peter},
  title = {{Riesz} transforms of a general {Ornstein--Uhlenbeck} semigroup},
  journal = {Calculus of Variations and Partial Differential Equations},
  volume = {60},
  pages = {135},
  year = {2021},
}

@article{DragicevicVolberg2006,
  author = {Dragi\v{c}evi\'c, Oliver and Volberg, Alexander},
  title = {{Bellman} function and dimensionless estimates of classical and {Ornstein--Uhlenbeck} {Riesz} transforms},
  journal = {Journal of Operator Theory},
  volume = {56},
  pages = {167--198},
  year = {2006},
}

@book{GilbargTrudinger2001,
  author = {Gilbarg, David and Trudinger, Neil S.},
  title = {Elliptic partial differential equations of second order},
  edition = {2},
  series = {Classics in Mathematics},
  publisher = {Springer},
  year = {2001},
  note = {Reprint of the 1998 edition},
  doi = {10.1007/978-3-642-61798-0},
}

@article{EskenazisIvanisvili2020,
  author = {Eskenazis, Alexandros and Ivanisvili, Paata},
  title = {Dimension independent {Bernstein--Markov} inequalities in {Gauss} space},
  journal = {Journal of Approximation Theory},
  volume = {253},
  pages = {105377},
  year = {2020},
}

@article{FabesGutierrezScotto1994,
  author = {Fabes, Eugene B. and Guti\'errez, Cristian E. and Scotto, Roberto},
  title = {Weak-type estimates for the {Riesz} transforms associated with the {Gaussian} measure},
  journal = {Revista Matem\'atica Iberoamericana},
  volume = {10},
  pages = {229--281},
  year = {1994},
}

@article{ForzaniScotto1998,
  author = {Forzani, Liliana and Scotto, Roberto},
  title = {The higher order {Riesz} transform for {Gaussian} measure need not be of weak type $(1,1)$},
  journal = {Studia Mathematica},
  volume = {131},
  pages = {205--214},
  year = {1998},
}

@incollection{ForzaniScottoUrbina2001,
  author = {Forzani, Liliana and Scotto, Roberto and Urbina, Wilfredo},
  title = {A simple proof of the {$L^p$} continuity of the higher order {Riesz} transforms with respect to the {Gaussian} measure {$\gamma_d$}},
  booktitle = {S\'eminaire de Probabilit\'es XXXV},
  series = {Lecture Notes in Mathematics},
  volume = {1755},
  pages = {162--166},
  publisher = {Springer},
  year = {2001},
}

@article{Freud1971,
  author = {Freud, G\'eza},
  title = {A certain inequality of {Markov} type},
  journal = {Doklady Akademii Nauk SSSR},
  volume = {197},
  pages = {790--793},
  year = {1971},
}

@article{Freud1977,
  author = {Freud, G\'eza},
  title = {On {Markov--Bernstein}-type inequalities and their applications},
  journal = {Journal of Approximation Theory},
  volume = {19},
  pages = {22--37},
  year = {1977},
}

@article{GarciaCuervaMauceriSjogrenTorrea1999,
  author = {Garc\'\i{}a-Cuerva, Jos\'e and Mauceri, Giancarlo and Sj\"ogren, Peter and Torrea, Jos\'e L.},
  title = {Higher-order {Riesz} operators for the {Ornstein--Uhlenbeck} semigroup},
  journal = {Potential Analysis},
  volume = {10},
  pages = {379--407},
  year = {1999},
}

@article{Gundy1986,
  author = {Gundy, Richard F.},
  title = {Sur les transformations de {Riesz} pour le semi-groupe d'{Ornstein--Uhlenbeck}},
  journal = {Comptes Rendus de l'Acad\'emie des Sciences, S\'erie I},
  volume = {303},
  pages = {967--970},
  year = {1986},
}

@article{GustafssonSakai1994,
  author = {Gustafsson, Bj\"orn and Sakai, Makoto},
  title = {Properties of some balayage operators, with applications to quadrature domains and moving boundary problems},
  journal = {Nonlinear Analysis},
  volume = {22},
  pages = {1221--1245},
  year = {1994},
}

@article{GutierrezSegoviaTorrea1996,
  author = {Guti\'errez, Cristian and Segovia, Carlos and Torrea, Jos\'e L.},
  title = {On higher {Riesz} transforms for {Gaussian} measures},
  journal = {Journal of Fourier Analysis and Applications},
  volume = {2},
  pages = {583--596},
  year = {1996},
}

@article{HilleSzegoTamarkin1937,
  author = {Hille, Einar and Szeg\H{o}, G\'abor and Tamarkin, Jacob D.},
  title = {On some generalizations of a theorem of {A. Markoff}},
  journal = {Duke Mathematical Journal},
  volume = {3},
  pages = {729--739},
  year = {1937},
}

@article{Janakiraman2004,
  author = {Janakiraman, Prabhu},
  title = {Weak-type estimates for singular integrals and the {Riesz} transform},
  journal = {Indiana University Mathematics Journal},
  volume = {53},
  pages = {533--555},
  year = {2004},
}

@book{Janson1997,
  author = {Janson, Svante},
  title = {{Gaussian} {Hilbert} spaces},
  series = {Cambridge Tracts in Mathematics},
  volume = {129},
  publisher = {Cambridge University Press},
  year = {1997},
}

@article{Kosov2026,
  author = {Kosov, Egor},
  title = {Dimension-free {Markov--Bernstein} inequalities for product measures},
  journal = {arXiv preprint arXiv:2606.13575},
  year = {2026},
}

@article{LarssonCohn2002,
  author = {Larsson-Cohn, Lars},
  title = {On the constants in the {Meyer} inequality},
  journal = {Monatshefte f\"ur Mathematik},
  volume = {137},
  pages = {51--56},
  year = {2002},
}

@book{Ledoux2001,
  author = {Ledoux, Michel},
  title = {The concentration of measure phenomenon},
  series = {Mathematical Surveys and Monographs},
  volume = {89},
  publisher = {American Mathematical Society},
  year = {2001},
}

@article{LevinLubinsky1994,
  author = {Levin, Eli and Lubinsky, Doron S.},
  title = {{$L^p$} {Markov--Bernstein} inequalities for {Freud} weights},
  journal = {Journal of Approximation Theory},
  volume = {77},
  pages = {229--248},
  year = {1994},
}

@article{LevinePeres2009,
  author = {Levine, Lionel and Peres, Yuval},
  title = {Strong spherical asymptotics for rotor-router aggregation and the divisible sandpile},
  journal = {Potential Analysis},
  volume = {30},
  pages = {1--27},
  year = {2009},
}

@article{Markov1890,
  author = {Markov, Andrey A.},
  title = {On a question by {D. I. Mendeleev}},
  journal = {Zapiski Imperatorskoi Akademii Nauk, St. Petersburg},
  volume = {62},
  pages = {1--24},
  note = {In Russian},
  year = {1890},
}

@article{MaoWangZhang2026Stratified,
  author = {Mao, Sheng-Chen and Wang, Yaojun and Zhang, Ye},
  title = {A {$p=2$} dichotomy for uniform {Riesz} transform bounds on stratified {Lie} groups},
  journal = {arXiv preprint arXiv:2608.20267},
  year = {2026},
}

@article{MauceriMeda2007,
  author = {Mauceri, Giancarlo and Meda, Stefano},
  title = {{BMO} and {$H^1$} for the {Ornstein--Uhlenbeck} operator},
  journal = {Journal of Functional Analysis},
  volume = {252},
  pages = {278--313},
  year = {2007},
}

@incollection{Meyer1984,
  author = {Meyer, Paul-Andr\'e},
  title = {Transformations de {Riesz} pour les lois gaussiennes},
  booktitle = {S\'eminaire de Probabilit\'es XVIII},
  series = {Lecture Notes in Mathematics},
  volume = {1059},
  pages = {179--193},
  publisher = {Springer},
  year = {1984},
}

@article{Mukherjee2026Drift,
  author = {Mukherjee, Suman},
  title = {Dimension-free weak $(1,1)$ estimates for {Riesz} transforms with constant drift},
  journal = {arXiv preprint arXiv:2609.21458},
  year = {2026},
}

@article{Muckenhoupt1969,
  author = {Muckenhoupt, Benjamin},
  title = {{Hermite} conjugate expansions},
  journal = {Transactions of the American Mathematical Society},
  volume = {139},
  pages = {243--260},
  year = {1969},
}

@article{Nguyen2022Skorohod,
  author = {Nguyen, Nhu N.},
  title = {Exponential tightness of a family of {Skorohod} integrals},
  journal = {Electronic Communications in Probability},
  volume = {27},
  number = {1},
  pages = {1--14},
  year = {2022},
  doi = {10.1214/21-ECP442},
}

@book {Nua06Malliavin,
    AUTHOR = {Nualart, David},
     TITLE = {The {M}alliavin calculus and related topics},
    SERIES = {Probability and its Applications (New York)},
   EDITION = {Second},
 PUBLISHER = {Springer-Verlag, Berlin},
      YEAR = {2006},
     PAGES = {xiv+382},
}

@article{OuyangSpectorStockdale2026,
  author = {Ouyang, Yuyuan and Spector, Daniel and Stockdale, Cody B.},
  title = {A dimension-free weak-type $(1,1)$ bound for the vector {Riesz} transform on {$\mathbb R^n$}},
  journal = {arXiv preprint arXiv:2608.18068},
  year = {2026},
}

@incollection{Pisier1988,
  author = {Pisier, Gilles},
  title = {{Riesz} transforms: a simpler analytic proof of {P.-A. Meyer's} inequality},
  booktitle = {S\'eminaire de Probabilit\'es XXII},
  series = {Lecture Notes in Mathematics},
  volume = {1321},
  pages = {485--501},
  publisher = {Springer},
  year = {1988},
}

@article{PronkVeraar2014,
  author = {Pronk, Mark and Veraar, Mark},
  title = {Tools for {Malliavin} calculus in {UMD} {Banach} spaces},
  journal = {Potential Analysis},
  volume = {40},
  pages = {307--344},
  year = {2014},
}

@article{ServadeiValdinoci2013,
  author = {Servadei, Raffaella and Valdinoci, Enrico},
  title = {{Lewy--Stampacchia} type estimates for variational inequalities driven by (non)local operators},
  journal = {Revista Matem\'atica Iberoamericana},
  volume = {29},
  pages = {1091--1126},
  year = {2013},
}

@incollection{Sjogren1983,
  author = {Sj\"ogren, Peter},
  title = {On the maximal function for the {Mehler} kernel},
  booktitle = {Harmonic analysis (Cortona, 1982)},
  series = {Lecture Notes in Mathematics},
  volume = {992},
  pages = {73--82},
  publisher = {Springer},
  year = {1983},
}

@article{SpectorStockdale2021,
  author = {Spector, Daniel and Stockdale, Cody B.},
  title = {On the dimensional weak-type $(1,1)$ bound for {Riesz} transforms},
  journal = {Communications in Contemporary Mathematics},
  volume = {23},
  pages = {2050072},
  year = {2021},
}

@article{Stein1983,
  author = {Stein, Elias M.},
  title = {Some results in harmonic analysis in {$\mathbb R^n$}, for {$n\to\infty$}},
  journal = {Bulletin of the American Mathematical Society},
  volume = {9},
  pages = {71--73},
  year = {1983},
}

@inproceedings{Stein1986,
  author = {Stein, Elias M.},
  title = {Problems in harmonic analysis related to curvature and oscillatory integrals},
  booktitle = {Proceedings of the International Congress of Mathematicians (Berkeley, 1986)},
  volume = {1},
  pages = {196--221},
  publisher = {American Mathematical Society},
  year = {1987},
}

@book{Zhu2012,
  author = {Zhu, Kehe},
  title = {Analysis on {Fock} spaces},
  series = {Graduate Texts in Mathematics},
  volume = {263},
  publisher = {Springer},
  year = {2012},
}

@article{Chen+26PicardHMC,
      title={Smoothed {P}icard {H}amiltonian {M}onte {C}arlo}, 
      author={Fan Chen and Sinho Chewi and Jianfeng Lu and Matthew S. Zhang},
      year={2026},
      journal={arXiv preprint arXiv:2609.06906},
}

@book {Gra14Fourier,
    AUTHOR = {Grafakos, Loukas},
     TITLE = {Classical {F}ourier analysis},
    SERIES = {Graduate Texts in Mathematics},
    VOLUME = {249},
   EDITION = {Third},
 PUBLISHER = {Springer, New York},
      YEAR = {2014},
     PAGES = {xviii+638},
}

@article {NewSha1966,
    AUTHOR = {Newman, Donald J. and Shapiro, Harold S.},
     TITLE = {Certain {H}ilbert spaces of entire functions},
   JOURNAL = {Bull. Amer. Math. Soc.},
  FJOURNAL = {Bulletin of the American Mathematical Society},
    VOLUME = {72},
      YEAR = {1966},
     PAGES = {971--977},
}

@incollection {Bor1978Tail,
    AUTHOR = {Borell, Christer},
     TITLE = {Tail probabilities in {G}auss space},
 BOOKTITLE = {Vector space measures and applications ({P}roc. {C}onf.,
              {U}niv. {D}ublin, {D}ublin, 1977), {I}},
    SERIES = {Lecture Notes in Math.},
    VOLUME = {644},
     PAGES = {73--82},
 PUBLISHER = {Springer, Berlin-New York},
      YEAR = {1978},
}

@book {Ver18HighDimProb,
    AUTHOR = {Vershynin, Roman},
     TITLE = {High-dimensional probability},
    SERIES = {Cambridge Series in Statistical and Probabilistic Mathematics},
    VOLUME = {47},
      NOTE = {An introduction with applications in data science,
              With a foreword by Sara van de Geer},
 PUBLISHER = {Cambridge University Press, Cambridge},
      YEAR = {2018},
     PAGES = {xiv+284},
}

@misc{NISTDLMF,
  author = {{National Institute of Standards and Technology}},
  title = {{NIST} digital library of mathematical functions},
  howpublished = {\url{https://dlmf.nist.gov/}},
  note = {Version 1.2.8, released September 15, 2026},
  year = {2026},
}

@article{LatalaOleszkiewicz1999,
  author = {Lata\l{}a, Rafa\l{} and Oleszkiewicz, Krzysztof},
  title = {{Gaussian} measures of dilatations of convex symmetric sets},
  journal = {Ann. Probab.},
  volume = {27},
  number = {4},
  pages = {1922--1938},
  year = {1999},
  doi = {10.1214/aop/1022874821},
}
}

\end{document}